\documentclass[oneside,11pt]{amsart}
\usepackage{amssymb,latexsym,amsmath,amsthm,enumitem,mathrsfs}
\usepackage[margin=1in]{geometry}
\usepackage[hidelinks]{hyperref}
\usepackage{fancyhdr}
\usepackage{color}
\usepackage{stmaryrd}
\usepackage{mathrsfs}
\usepackage{lineno}
\usepackage{diagbox}
\usepackage{array}
\usepackage{tikz}
\usetikzlibrary{arrows}
\usetikzlibrary{positioning}
\usepackage{float}
\usepackage{pgfplots}
\usepackage[normalem]{ulem}
\newcolumntype{C}[1]{>{\centering\let\newline\\\arraybackslash\hspace{0pt}}m{#1}}

\numberwithin{equation}{section}

\makeatletter
\def\@tocline#1#2#3#4#5#6#7{\relax
  \ifnum #1>\c@tocdepth 
  \else
    \par \addpenalty\@secpenalty\addvspace{#2}%
    \begingroup \hyphenpenalty\@M
    \@ifempty{#4}{%
      \@tempdima\csname r@tocindent\number#1\endcsname\relax
    }{%
      \@tempdima#4\relax
    }%
    \parindent\z@ \leftskip#3\relax \advance\leftskip\@tempdima\relax
    \rightskip\@pnumwidth plus4em \parfillskip-\@pnumwidth
    #5\leavevmode\hskip-\@tempdima
      \ifcase #1
       \or\or \hskip 1em \or \hskip 2em \else \hskip 3em \fi%
      #6\nobreak\relax
    \hfill\hbox to\@pnumwidth{\@tocpagenum{#7}}\par
    \nobreak
    \endgroup
  \fi}
\makeatother

\newtheorem{letterthm}{Theorem}

\newtheorem{thm}{Theorem}[section]
\newtheorem*{thm*}{Theorem}

\newtheorem{lemma}[thm]{Lemma}
\newtheorem{property}[thm]{Property}
\newtheorem{cor}[thm]{Corollary}
\newtheorem{conjecture}[thm]{Conjecture}

\theoremstyle{remark}
\newtheorem{remark}[thm]{Remark}

\newcommand{\ep}{\varepsilon}

\newcommand{\ndiv}{\nmid}

\newcommand{\bstack}[2]{#1 \atop #2}

\newcommand{\maps}{\rightarrow}
\newcommand{\intersect}{\cap}

\newcommand{\union}{\cup}

\newcommand{\rk}{{\rm rk }}

\newcommand{\al}{\alpha}
\newcommand{\be}{\beta}

\newcommand{\ga}{\gamma}
\newcommand{\del}{\delta}
\newcommand{\Del}{\Delta}

\newcommand{\om}{\omega}

\newcommand{\sig}{\sigma}
\newcommand{\lam}{\lambda}

\newcommand{\Mcal}{\mathcal{M}}

\newcommand{\Dbf}{\mathbf{D}}

\newcommand{\beq}{\begin{equation}}
\newcommand{\eeq}{\end{equation}}
\input{melanie.sty}

\newcommand{\Q}{\mathbb{Q}}
\newcommand{\R}{\mathbb{R}}
\newcommand{\F}{\mathbb{F}}

\newcommand{\Z}{\mathbb{Z}}

\newcommand{\Cbf}{\mathbf{C}}

\newcommand{\pfrak}{\mathfrak{p}}

\newcommand{\Cl}{\mathrm{Cl}}

\newcommand{\Bscr}{\mathscr{B}}

\newcommand{\Escr}{\mathscr{E}}
\newcommand{\Fscr}{\mathscr{F}}

\newcommand{\Lscr}{\mathscr{L}}

\newcommand{\Gal}{\mathrm{Gal}}
\newcommand{\Disc}{\mathrm{Disc}\,}

\newcommand{\GL}{\mathrm{GL}}

\newcommand{\Ind}{\mathrm{Ind}}
\newcommand{\Nm}{\mathrm{Nm}}

\definecolor{pink}{rgb}{1,.2,.6}
\definecolor{orange}{rgb}{0.7,0.3,0}
\definecolor{blue}{rgb}{.2,.6,.75}
\definecolor{green}{rgb}{.4,.7,.4}
\definecolor{purple}{RGB}{127,0,255}

\newcommand{\xtra}[1]{}

\newcommand{\exendnote}[1]{}

\newcommand{\apphide}[1]{}

\begin{document}

\title[Pierce, Turnage-Butterbaugh, Wood]{On a conjecture for $\ell$-torsion\\ in class groups  of number fields:\\ from the perspective of moments}

\author[Pierce]{Lillian B. Pierce}
\address{Department of Mathematics, Duke University, 120 Science Drive, Durham, NC 27708 USA}
\email{pierce@math.duke.edu}

\author[Turnage-Butterbaugh]{Caroline L. Turnage-Butterbaugh}
\address{Department of Mathematics \& Statistics, Carleton College, 1 North College Street, Northfield, MN 55057 USA}
\email{cturnageb@carleton.edu}

\author[Wood]{Melanie Matchett Wood}
\address{Department of Mathematics,
Harvard University,
1 Oxford Street,
Cambridge, MA 02138 USA}  
\email{mmwood@math.harvard.edu}




\begin{abstract}
It is conjectured that within the class group of any number field, for every integer $\ell \geq 1$, the $\ell$-torsion subgroup is very small (in an appropriate sense, relative to the discriminant of the field). In nearly all settings, the full strength of this conjecture remains open, and even partial progress is limited. Significant recent progress toward average versions of the $\ell$-torsion conjecture has relied crucially  on counts for number fields, raising interest in how these two types of question relate. In this paper we make explicit the quantitative relationships between the $\ell$-torsion conjecture and other well-known conjectures: the Cohen-Lenstra heuristics, counts for number fields of fixed discriminant, counts for number fields of bounded discriminant (or related invariants), counts for elliptic curves with fixed conductor. All of these considerations reinforce that we expect the $\ell$-torsion conjecture to be true, despite limited progress toward it. Our perspective focuses on the relation between pointwise bounds, averages, and higher moments, and demonstrates the broad utility of the ``method of moments.'' 

\end{abstract}

\maketitle

\section{Introduction}

Associated to each number field $K/\Q$ is the ideal class group $\Cl_{K}$, a finite abelian group
that encodes information about arithmetic in $K$, and can also be seen as the Galois group of the maximal unramified abelian extension of $K$. Correspondingly, each number field $K$ has a class number, defined to be the cardinality $|\Cl_K|$ of the class group.
Given a field $K$, it is natural to ask about the size and  structure of the class group; moreover, as $K$ varies over an infinite set, or ``family,'' of fields, it is natural to ask how the corresponding class groups are distributed. 
 Interest in such questions has a long history, going back to C.F. Gauss's class number conjecture, early attempts at proving Fermat's Last Theorem, and Dirichlet's development of the class number formula. In particular, although class numbers arise initially from an algebraic construction, the class number formula identifies them with values of $L$-functions and hence to core questions in analytic number theory, including the Generalized Riemann Hypothesis.

Our focus in this paper is  a well-known conjecture for $\ell$-torsion in class groups, which, in the strongest form in which it has been proposed, remains stubbornly out of reach.
For each integer $\ell \geq 1$, the  $\ell$-torsion subgroup of $\Cl_K$ is defined by 
\[
 \Cl_{K}[\ell]:= \{ [\mathfrak{a}]\in\Cl_{K} : \ell [\mathfrak{a}]= \mathrm{Id} \},
\]
in which we write the class group in additive notation.
Thus for example, if the $\ell$-torsion is trivial, the class number is not divisible by $\ell$.

How big is $\Cl_K[\ell]$?
 For any number field $K/\Q$ of degree $n$ and discriminant of absolute value $D_K$, we may trivially bound the $\ell$-torsion subgroup for any integer $\ell \geq 1$ by the size of the full class group, so that\footnote{We will use Vinogradov's notation: $A \ll B$ denotes that there exists a constant $C$  such that $|A| \leq C B,$  and $A\ll_\kappa B$ denotes  that  $C$ may depend on $\kappa$.}
\beq\label{trivial_bound}
1 \leq |\Cl_K[\ell]| \leq |\Cl_K| \ll_{n,\ep} D_K^{1/2+\ep},
 \eeq
 for  all $\ep>0$ arbitrarily small, according to Landau's upper bound for the class number via the Minkowski bound (see e.g. \cite[Theorem 4.4]{Nar80}).
But conjecturally, a much stronger upper bound should hold.  
\begin{conjecture}[$\ell$-torsion Conjecture]\label{conj_class}
\item For every integer $n \geq 2$, every field extension $K/\Q$ of degree $n$ satisfies the property that for all primes $\ell \geq 2$ and every $\ep>0$,
\beq\label{CL_ell_bound}
|\Cl_{K}[\ell]| \ll_{n,\ell,\ep}D_K^\ep.
\eeq
\end{conjecture}
This conjecture  was first raised as a question of Brumer and Silverman \cite[Question $\mathrm{CL}(\ell,d)$]{BruSil96} in the context of counting elliptic curves with fixed conductor; we will return to this original motivation, and a possible more precise rate of growth, in \S \ref{sec_EC}.  It is known to have implications for many other problems: for example, proving this upper bound would imply strong bounds for the ranks of elliptic curves (see \cite[\S 1.2]{EllVen07}) and for counts of certain number fields (see Remarks \ref{remark_Alberts} and  \ref{remark_Hasse}). See also Zhang \cite[Conjecture 3.5]{Zha05} in the context of equidistribution of CM points on Shimura varieties,  Duke \cite[p. 166]{Duk98} in consideration of counting number fields, and
Belolipetsky and Lubotzky \cite[Thm. 7.5]{BelLub17}, who show an equivalence between certain more precise forms of Conjecture \ref{conj_class} and a conjecture on counting non-uniform lattices in semisimple Lie groups.
We have stated the conjecture here for $\ell$ prime, as for composite $\ell$ the result follows from the prime cases; see \S \ref{sec_composite}.

Conjecture \ref{conj_class} is known to be true in only one case: degree $n=2$ and the prime $\ell=2$; this follows from the genus theory of Gauss \cite{Gau01}, which shows that for a fundamental discriminant of absolute value $D_K$, we have $|\Cl_K[2]| \leq 2^{\om(D_K)-1}$, 
 where $\om(D_K)$ denotes the number of distinct prime divisors of $D_K$.  
 Since $\om(m) \ll \log m/ \log \log m$ (see e.g. \cite[Ch. 22 \S 10]{HarWri08}), it follows that for any quadratic field $K$, for every $\ep>0$,
 \beq\label{Gauss_2}
 |\Cl_K[2]| \ll_\ep D_K^\ep .
 \eeq

Recent progress has focused on verifying results toward Conjecture \ref{conj_class} by making small improvements on (\ref{trivial_bound}), either for all number fields of a fixed degree, or for ``almost all'' fields in a family of number fields of fixed degree. Results of the first type include \cite{HelVen06,Pie05,Pie06,EllVen07,BSTTTZ17}. Results of the second type, which can also be thought of as results  ``on average'' over a family of fields, include \cite{DavHei71,Bha05,HBP17, EPW17,PTBW20,FreWid17,FreWid18x}.  (For an overview of recent work, see Section \ref{sec_previous}.)

The apparent difficulty of verifying Conjecture \ref{conj_class}, for a fixed prime $\ell$ and all fields of a fixed degree $n$, in any case other than that due to Gauss, leads to a question: is the strong pointwise version of Conjecture \ref{conj_class} stated above too much to expect?

In this paper we describe several explicit  motivations for why we expect Conjecture \ref{conj_class} to hold. The key philosophy we employ is the ``method of moments,'' surveyed in \S \ref{sec_moments}: this is the principle that in certain  settings, if one can control arbitrarily high moments of a function, then one can deduce that there is not even one violation of a pointwise upper  bound by the function. 
In particular we show \emph{quantitatively}  how Conjecture \ref{conj_class} follows  from three other well-known conjectures in number theory: the Cohen-Lenstra heuristics; the discriminant multiplicity conjecture for counting number fields with fixed discriminant; and a generalization of the Malle conjecture for counting number fields with bounded discriminant. 
 Furthermore, we quantify implications of these conjectures for counting elliptic curves with fixed conductor, and we prove new unconditional counts for elliptic curves.
 
Many recent works on $\ell$-torsion in class groups have focused on average results over families of fields, and this type of work has been significantly constrained by the difficulty of counting number fields.  Our investigations in this paper confirm that counting fields is not merely a technical difficulty one encounters when proving average results, but in fact lies at the heart of understanding $\ell$-torsion in class groups.

We now give an overview of  this paper, organized according to four main  themes. 

\subsection{Cohen-Lenstra and Cohen-Martinet heuristics (\S \ref{sec_CLM_predictions})}\label{sec_CL_intro}
 First, we make quantitatively precise how Conjecture \ref{conj_class} follows from  the Cohen-Lenstra \cite{CohLen84} and Cohen-Martinet \cite{CohMar90} heuristics.
  In fact, we show  that even a collection of weaker upper bounds for moments, implied by the Cohen-Lenstra-Martinet heuristics, would imply Conjecture \ref{conj_class}. We frame our first result as a statement about a ``family''  $\Fscr$ of degree $n$ field extensions of $\Q$; we will at the moment leave the specifications of such a family rather vague, but this could indicate that we have fixed not only the degree and the Galois group of the Galois closure, but also the signature, certain local conditions, and possibly certain ramification restrictions (see \S \ref{sec_conventions} for precise conventions).
  
    Let $\Fscr$ denote a family of fields $K/\Q$ of degree $n$ and let $\Fscr(X)$ denote those with $0<D_K \leq X$, where $D_K = |\Disc K/\Q|$. 
 Consider for any fixed real number $k \geq 1$, integer $\ell \geq 1$, and real number $\al \geq 1$ the statement that  
 for all $X \geq 1$,
  \beq\label{CL_k_upper}
 \sum_{K \in \Fscr(X)} |\Cl_K[\ell]|^k \ll _{n,\ell,k,\al} |\Fscr(X)|^{\al}.
 \eeq

   \begin{thm}\label{thm_CLM_torsion}
 Let $n,\ell$ be fixed.  Suppose that for a fixed value $\al \geq 1$,  (\ref{CL_k_upper}) is known for all integers $k \geq 1$. Then for every $\ep>0$ we have $|\Cl_K[\ell]| \ll_{\ep} D_K^{\ep}$ for every field $K \in \Fscr$.
  \end{thm}
  
 This principle can be applied to any family $\Fscr$ for which appropriate moment bounds are known. Here we record that Theorem \ref{thm_CLM_torsion} immediately implies that for fixed $n,\ell$, uniform control of arbitrarily high $k$-th moments  is sufficient to deduce Conjecture \ref{conj_class} in its full strength:
  \begin{cor}\label{cor_CL}
Let $n\geq 1$ be fixed and let $\Fscr$ be taken to be the family of all degree $n$ extensions of $\Q$.
For each fixed integer $\ell \geq 1$, the conclusion of Conjecture \ref{conj_class} for $\ell$-torsion for all $K \in \Fscr$ would follow from knowing (\ref{CL_k_upper}) holds for that fixed $n,\ell$ and a certain fixed $\al \geq 1$,  for arbitrarily large  $k$.
  \end{cor}

  See Remark \ref{remark_FK} for an example of how to apply this moment approach to $4$-torsion in quadratic fields; while this case may be deduced by other means (\S \ref{sec_composite}), this is nevertheless an illustration of the principle.

\subsection{Discriminant multiplicity conjecture (\S \ref{sec_DMC})}
In the second theme, we will show quantitatively how Conjecture \ref{conj_class} follows from a well-known folk conjecture about counting number fields with fixed degree and fixed discriminant. 
By Hermite's finiteness theorem, for any fixed integer $D \geq 1$ there are finitely many extensions $K/\Q$ of degree $n$ with $D_K=D$ (see e.g. \cite[\S 4.1]{Ser97}). The question remains: how many? 
Duke \cite[\S 3]{Duk98} and Ellenberg and Venkatesh \cite[Conjecture 1.3]{EllVen05} make the following prediction, which we call the discriminant multiplicity conjecture:
 \begin{conjecture}[Discriminant Multiplicity Conjecture]\label{conj_DMC}
For each $n \geq 2$, for every $\ep>0$  
and for every integer $D \geq 1$, at most $\ll_{n,\ep} D^{\ep}$ fields $K/\Q$ of degree $n$ have $D_K=D$.
\end{conjecture}
Conjecture \ref{conj_DMC} is known to be true for quadratic fields ($n=2$) but is not known in full for any $n \geq 3$. 
See \S \ref{sec_DMC} for a summary of known results toward the discriminant multiplicity conjecture.

 Ellenberg and Venkatesh have noted  that Conjecture \ref{conj_DMC} implies Conjecture \ref{conj_class} (see \cite[p. 164]{EllVen05}). In Theorem \ref{thm_fields_cl}, we make this explicit by exhibiting quantitatively how even a partial result toward Conjecture \ref{conj_DMC} implies progress toward Conjecture \ref{conj_class}. To do so, we define two notations:
 \begin{property}[Property $\Dbf_n(\varpi)$]\label{property_D}
Let $n \geq 2$ be fixed. We say that property ${\bf D}_n(\varpi)$ holds if  for every $\ep>0$ 
and for every fixed integer $D>1$, at most $\ll_{n,\ep} D^{\varpi+\ep}$ fields $K/\Q$ of degree $n$ have $D_K=D$.
\end{property} 

  \begin{property}[Property $\Cbf_{n,\ell}(\Del)$]\label{property_C}
Let $n, \ell \geq 1$ be fixed. We say that property ${\bf C}_{n,\ell}(\Del)$ holds if for every $\ep>0$ 
and for all fields $K/\Q$ of degree $n$, we have
 $|\Cl_{K}[\ell]| \ll_{n,\ell,\ep} D_K^{\Del+\ep }.$
\end{property} 

Now we make the relationship between $\Dbf_n(\varpi)$ and the $\ell$-torsion conjecture quantitatively precise.

   \begin{thm}\label{thm_fields_cl}
Fix an integer $n \geq 2$ and any prime $\ell$. If Property ${\bf D}_{n\ell} (\Del)$ holds for a real number $\Del \geq 0$, then Property ${\bf C}_{n,\ell}(\ell \Del)$ holds. 
 \end{thm}
   Note that the exponent we gain for the $\ell$-torsion bound is significantly weaker (that is, larger) than the exponent we assume for counting fields with fixed discriminant; in this formulation, as $\ell$ increases, we need to be increasingly good at counting the relevant fields, to preserve the same quality of bound on $\ell$-torsion. 
Theorem \ref{thm_fields_cl} has an immediate corollary:
  \begin{cor}
Let $n \geq 1$ be fixed. Then for a fixed prime $\ell$, the statement of Conjecture \ref{conj_class} holds for all fields of degree $n$, if  Conjecture \ref{conj_DMC} is known for degree $n\ell$.
  \end{cor}

\begin{remark}\label{remark_deg_n} We highlight here a difficulty of this approach; if one fixes $n$ and wants to use this route to prove the $\ell$-torsion conjecture for all $\ell$ for fields of degree $n$, then one would need to prove the discriminant multiplicity conjecture for infinitely many degrees. The recent work of the authors in \cite{PTBW20} to some extent eliminates this inefficiency; in that work, once the degree $n$ is fixed, (weak) results are proved for $\ell$-torsion 
for all $\ell$   by proving weak results toward the discriminant multiplicity conjecture only for degree $n$; see \S \ref{sec_review_avg} for further remarks.
\end{remark}

\subsection{Generalized Malle Conjecture (\S \ref{sec_gen_Malle})}
In the third theme, we examine how Conjecture \ref{conj_DMC}, and hence Conjecture \ref{conj_class}, follows from a generalized Malle conjecture for counting number fields ordered via other invariants than the discriminant.  
This observation is due to Ellenberg and Venkatesh \cite[Prop. 4.8]{EllVen05}, and we state it here only loosely (see \S \ref{sec_gen_Malle} for a precise statement):
 
 \begin{letterthm}[{\cite[Prop. 4.8]{EllVen05}}]\label{T:EVA}
If a generalized Malle conjecture (see (\ref{Malle_EV})) holds for all groups $G$, then Conjecture \ref{conj_DMC} holds, and hence Conjecture \ref{conj_class} holds.
 \end{letterthm}
We highlight this theorem of Ellenberg and Venkatesh because it uses not only the method of moments  but also a clever  strategy of reinterpreting a $k$-th moment of an object  as an average of $k$-fold objects, so that one can (conjecturally) obtain bounds for arbitrarily high moments even if one only  assumes bounds for averages (i.e. first moments). (This leads to speculation on whether it could be advantageous to reinterpret $k$-th moments of $\ell$-torsion as averages of $\ell'$-torsion over certain $k'$-fold objects, but any such precise formulation is so far elusive.)  
\begin{remark}\label{remark_Alberts}
Note also that Alberts \cite{Alb18} has shown that Conjecture \ref{conj_class} implies the generalized Malle conjecture that we give below in (\ref{Malle_EV}), providing a converse to Theorem~\ref{T:EVA}.
\end{remark}

\subsection{Counting elliptic curves with fixed conductor (\S \ref{sec_EC})}
In our fourth theme, we observe that for every $k \geq 1$, a $k$-th moment estimate for counting elliptic curves with fixed conductor can be deduced from any of the conjectures mentioned above. In addition, unconditionally, we derive improved bounds for counting elliptic curves with fixed conductor.

To be precise, given a positive integer $q$, let $E(q)$ denote the number of isogeny classes of elliptic curves over $\Q$ with conductor $q$. Brumer and Silverman   have put forward the following conjecture:

\begin{conjecture}[{\cite[Conj. 3]{BruSil96}}]\label{conj_EC}
For every $q \geq 1$, for all $\ep>0$,
\[ E(q) \ll_\ep q^{\ep}. \]
\end{conjecture}
(For clarity, we note that an equivalently strong conjecture is expected to hold if we instead count the number $E'(q)$ of isomorphism classes of elliptic curves over $\Q$ with conductor $q$; the statement of our results below will hold for both $E(q)$ and $E'(q)$, up to a constant factor.)
Brumer and Silverman proved the first result toward Conjecture \ref{conj_EC}, namely $E(q) \ll_\ep q^{1/2+\ep}$, by dominating this count by $3$-torsion in certain quadratic fields; by combining aspects of  \cite[Thm. 4.5]{HelVen06} and  \cite{EllVen07} this can now be improved to 
\beq\label{Eq_EV}
E(q) \ll_\ep q^{2 \be/ 3 \log 3 + \ep} \ll_\ep q^{0.1688...+\ep} 
\eeq
for an explicit constant $\be = 0.2782...$;
 see \S \ref{sec_EC} for an explication of this.
 We observe in Theorem \ref{thm_EC_6} that Conjecture \ref{conj_EC} would follow from counting sextic fields of fixed discriminant. But our main focus is on a moment statement related to Conjecture \ref{conj_EC}.

Brumer and Silverman's argument dominates $E(q)$ pointwise by a sup-norm of $|\Cl_{K}[3]|$ over quadratic fields; this comparison has been sharpened further in \cite{HelVen06}, still using a supremum over a  set of certain quadratic fields. Thus proving an upper bound for $3$-torsion for almost all such fields (or  for an average, or for a fixed moment) does not immediately provide an analogous result for $E(q)$ (see  Remark \ref{remark_EC}). 
The content of Duke and Kowalski \cite[Prop. 1]{DukKow00} is that a refined argument shows that nevertheless, average counts for $E(q)$ are dominated by average counts for $3$-torsion in quadratic fields. (See \cite{FNT92} for an analogue when the curves are ordered by discriminant rather than conductor.) 

An upper bound for an average $\sum E(q)$ yields an upper bound for the number of $q$ for which $E(q)$ can violate a certain pointwise upper bound; a comparable upper bound for a higher moment $\sum E(q)^k$ sharpens the upper  bound on the number of  possible exceptions. 
This motivates us to consider higher moments.
We build on \cite{BruSil96,DukKow00,HelVen06} to prove that  the $k$-th moment of $3$-torsion in quadratic fields dominates 
the $(\ga k)$-th moment of $E(q)$, for a certain numerical constant $\ga \approx 2$ specified  precisely below.
\begin{thm}\label{thm_EC}
Let $\Fscr_2(X)$ denote the family of quadratic fields  with discriminant of absolute value at most $X$.
Let $\ga = \log 3 / (2\be) = 1.9745...$ be the numerical constant determined by $\be= \be(0) = 0.2782...$ as defined in \cite[Thm. 3.8]{HelVen06}.
Let $k \geq 1/\ga =0.5064...$ be a fixed real number, and assume that there exists $\Theta_k \geq 1$ such that for every $\ep>0$ 
and for every $X \geq 1$, 
\beq\label{assume_k_moment_Cl3}
 \sum_{K \in \Fscr_2(X)} |\Cl_K[3]|^k  \ll_{k,\ep} X^{\Theta_k + \ep}.
 \eeq
Then for that $k$, for every $\ep>0$  
and  for all $Q \geq 1$,
\beq\label{get_2k_moment_Eq}
 \sum_{q \leq Q} E(q)^{\ga k}  \ll_{k,\ep}Q^{\Theta_k+\ep}.
 \eeq

\end{thm}

Notice that the conclusion for $E(q)$ is stronger  than the assumption for $|\Cl_K[3]|$, since the power   in the moment   has   nearly doubled in size. In particular,   (\ref{assume_k_moment_Cl3}) shows that 
\[ \# \{ K \in \Fscr_2(X) : X/2 < D_K \leq X,  |\Cl_K[3]| \geq D_K^\varpi\} \ll X^{\Theta_k - k \varpi },\]
while we see from (\ref{get_2k_moment_Eq}) that 
\[ \# \{ Q/2 \leq q \leq Q :  E(q) \geq q^\varpi\} \ll Q^{\Theta_k - k \ga \varpi }.\]

By combining Theorem \ref{thm_EC} with moment bounds of the form (\ref{assume_k_moment_Cl3}) for $3$-torsion in quadratic fields due to Heath-Brown and the first author in \cite[Cor. 1.4]{HBP17} we immediately obtain new unconditional bounds for moments of $E(q)$:
\begin{cor}\label{cor_EC}
For each $q \geq 1$, let $E(q)$ denote the number of isogeny classes of elliptic curves over $\Q$ with conductor $q$. Let  $\ga = 1.9745...$ be the numerical constant as above. Then for every $Q \geq 1$, for all $\ep>0$,
\[ \sum_{q \leq Q} E(q)^{2 \ga }  \ll_\ep Q^{23/18+\ep}.\]
More generally, for all real $1 \leq k \leq 4$,
\[ \sum_{q \leq Q} E(q)^{\ga k}  \ll_\ep Q^{(5k+13)/18+\ep},\]
and for all real $k \geq 4$,  
\[ \sum_{q \leq Q} E(q)^{\ga k}  \ll_\ep Q^{(2k+3)/6+\ep}.\]
\end{cor}
These results improve on all results implicitly derivable in previous literature; see \S \ref{sec_EC} for details.

\subsection{Concluding sections \S \ref{sec_previous} and \S \ref{sec_appendices}}
In \S \ref{sec_previous} we briefly survey certain recent works toward Conjecture \ref{conj_class}, with a focus on the relation to the Generalized Riemann Hypothesis. In contrast to the impact of the other major conjectures we have mentioned, GRH does not appear to imply Conjecture \ref{conj_class} directly, due to an interesting interaction with another set of conjectures, about number fields with generating elements of small height. Finally, in \S \ref{sec_appendices} we include brief appendices with further remarks about moment bounds and one special case of the discriminant multiplicity conjecture.

\subsection{Notational conventions on families of fields}\label{sec_conventions}
A ``family'' (set) of fields is formally defined as follows. 
Given an integer $n \geq 1$ and a fixed transitive subgroup $G \subseteq S_n$, for every integer $X \geq 1$ we can define a family $\Fscr(X)$ to be $\{ K/\Q : \deg{K/\Q} = n, \Gal(\tilde{K}/\Q) \simeq G, D_K= |\Disc K/\Q| \leq X \},$
where all $K$ are in a fixed algebraic closure $\overline{\Q}$, $\tilde{K}$ is the Galois closure of $K$ over $\Q$,  the Galois group is considered as a permutation group on the $n$ embeddings of $K$ in $\overline{\Q}$, and the isomorphism with $G$ is one of permutation groups.
We let $\Fscr =\Fscr(\infty)$. It is furthermore possible to impose additional restrictions on the definition of the family, such as fixing the signature, certain local conditions, or ramification restrictions. In some cases where we wish to be more vague (for example, considering all degree $n$ extensions), we will just refer to a family $\Fscr$ and the subset $\Fscr(X)$ of those fields $K \in \Fscr$ with $D_K \leq X$.
Given a family $\Fscr$, we will say that a set $E \subset \Fscr$ is a subset of  density zero if 
\[  \frac{|E \intersect \Fscr(X)|} {|\Fscr(X)|}\maps 0 \quad \text{as $X \maps \infty$}.\]
If a property holds for all fields in a family except those lying in a possible exceptional set of density zero, we say the property holds for ``almost all'' fields in the family.

\section{The method of moments}\label{sec_moments}
We start with a general philosophy for why one can expect control of arbitrarily high moments of a function to yield good control on the function's values. On a ($\sig$-finite) measure space $(\Mcal,\mu)$ such as $\R$ with Lebesgue measure or $\Z$ with counting measure, given a complex-valued measurable function $f$, define for every real $\al>0$ the distribution function $
\lam(\al) = \mu( \{ x: |f(x)| > \al \})$. 
Then for any $\al>0$,
\beq\label{int_id}
 \al^{p} \lam(\al) \leq \int_{\{|f| >\al\}} |f(x)|^p d\mu \leq \int_\Mcal |f(x)|^p d\mu = \|f\|_{L^p(\Mcal,\mu)}^p .
\eeq
Suppose that it is known that for a sequence of $p$ growing arbitrarily large we have $\|f\|_{L^p}^p \leq C$ for a fixed constant $C$, uniformly in $p$. Assuming this, one intuitively sees that for fixed $\al > 1$, as $p \maps \infty$ through such a sequence, $\lam (\al)$ must be  vanishingly small in (\ref{int_id}). This is the principle by which uniform control of arbitrarily high moments  implies pointwise ($\mu$-a.e.) bounds. (By a ``moment'' we refer to the $p$-th power of the $L^p$ norm, and not the $L^p$ norm itself. Using the notation of the distribution function $\lam(\al)$ defined above, we have 
\[ \int_\Mcal |f(x)|^p d\mu = - \int_0^\infty \al^p d\lam(\al) = p\int_0^\infty  \al^{p-1} \lam(\al)d\al,\]
which are all  expressions for how ``moments'' are defined in different contexts, motivating the terminology.)

Let us see more concretely how this philosophy applies in a model of the discrete setting that is our concern in this paper,  by taking counting measure on $\Z$ (denoted by $|E|$ for a set $E \subset \Z$) and considering integer-valued functions of finite support (that is, vanishing aside from at finitely many integers). We will consider a restricted norm $\| f\|_{\ell^p(A_R)}^p = \sum_{n \in A_R} |f(n)|^p$ where $A_R$ is the dyadic interval   $A_R=\{n: R \leq n \leq 2R\}$ with $R \geq 1$. Now we will aim to show not that $f$ is uniformly bounded but that it does not grow too quickly. 

For any $1 \leq p < \infty$ and any $\Del>0$ the following inequality holds:
\beq\label{moment1}
   R^{\Del p} | \{ n \in A_R: |f(n)| > n^\Del\}|<\sum_{\{n \in A_R : |f(n)| > n^\Del\}} |f(n)|^p \leq \sum_{n \in A_R} |f(n)|^p =\|f\|_{\ell^p(A_R)}^p.
\eeq
Assume that for a sequence of $p \geq 1$,   for all $R \geq 1$ we have 
\beq\label{bound_assumed_uniform}
\|f\|_{\ell^p(A_R)}^p \leq |A_R|^\be
\eeq
 for some fixed $\be>0$. Then  we may conclude that for each such $p$, for all $R \geq 1$, we have 
\beq\label{moment2}
 | \{ n \in A_R: |f(n)| > n^\Del\}| \leq |A_R|^\be R^{-\Del p} \leq R^{\be - \Del p} .
 \eeq
In particular, if the sequence of $p$ for which we know this grows arbitrarily large, we can for each $R$ take $p$ sufficiently large that $R^{\be - \Del p}<1$ so that the set on the left-hand side of (\ref{moment2}) must be empty, that is, for all $n \in A_R$ we have $|f(n)| \leq n^\Del$. 
(Somewhat more nuanced arguments are required if the bound in (\ref{bound_assumed_uniform}) is not uniform in $p$; see the proof of Theorem \ref{thm_CLM_torsion} and the comments in \S \ref{sec_meas_space}.)

We will adapt this method of moments first to the setting in which our ``discrete space'' is a set of fields, and the function $f$ measures the size of the $\ell$-torsion subgroup (\S \ref{sec_CLM_predictions}); and later, to the setting in which the discrete space is $\mathbb{N}$ and $f$ counts the number of fields with a specified discriminant (\S \ref{sec_gen_Malle}).
These ideas may be adapted to many other settings, such as other functions mapping fields (ordered by some invariant) to non-negative real numbers;
for detailed remarks from the perspective of $L^p$ spaces, which may be helpful in such generalizations, see the appendix in \S \ref{sec_meas_space}.

Note that in practice, even if one cannot prove uniform control of arbitrarily high moments, it can sometimes be useful to prove an upper bound for one particular moment, such as the second moment. If an upper bound for an average shows that a certain pointwise bound must hold for almost all elements in a family (that is, for all but an exceptional family of density zero), then proving an appropriate upper bound on a higher moment can force a better upper bound on the cardinality of the possible exceptional family. See Remark \ref{remark_moment_ae} for an instance of this; this also motivates the interest for Theorem \ref{thm_EC} and Corollary \ref{cor_EC} on elliptic curves.

The method of moments has a long history   and wide applicability in number theory; we mention two additional settings. First, the method of moments shows that the Lindel\"of Hypothesis for the Riemann zeta function would be a consequence of upper bounds for arbitrarily large even moments of $\zeta(1/2+it)$ along the critical line. Precisely, arguments such as \cite[Thm. 13.2]{Tit86} show that the statement 
\[ \zeta(1/2 + it) \ll_\ep t^\ep \qquad \text{for all $\ep>0$, for all $t \in \R$}\]
is equivalent to the statement that for all $k=1,2,3,\ldots,$ 
\[ \frac{1}{T} \int_{1}^T |\zeta(1/2 + it)|^{2k} dt  \ll_{k,\ep} T^\ep \qquad \text{for all $\ep>0$, for all $T \in \R$}.\]
(We remark that in this context of Lebesgue measure on $\R$, rather than a discrete counting measure such as we will consider, one requires some additional consideration to ensure that a result holds for all $t$ and not just a.e. $t$. For this result on the zeta function, this comes from an integral expression for the analytic continuation, and the resulting notion of the derivative of $\zeta$ in the critical strip, so that any large value of $\zeta(1/2+it)$ would imply that $\zeta$ also takes large values on a positive measure neighborhood.)

Second, the principle of the method of moments can be used to deduce  the Weil-Deligne bound for character sums of prime moduli from another Riemann Hypothesis, that of $L$-functions of algebraic varieties over finite fields. This area is very deep (see e.g. \cite{Kow10} for an overview). Even given Deligne's proof \cite{Del74,Del80} of the Weil conjectures (which includes deductions for certain one-variable additive character sums), deducing the desired (square-root cancellation) bounds for complete character sums in higher dimensions is difficult (the subject, for example, of many celebrated works of N. Katz, which fully employ tools of algebraic geometry). In contrast, ``elementary'' presentations proceed by the philosophy we have encapsulated above as the method of moments. 
The idea at the heart of this approach is, roughly speaking, that the  character sum $S$ modulo a prime $p$ may be embedded inside a family of character sums $S_\nu$ over $\F_{p^\nu}$ for all $\nu \geq 1$. On the one hand, for each $\nu$ the sum $S_\nu$ may be interpreted as a sum of roots associated to a  rational function (the zeta function), which are themselves $\nu$-powers of the roots associated to the original sum $S$.  On the other hand, for each $\nu$ the sum $S_\nu$ may be interpreted in terms of counting points on a variety over $\F_{p^\nu}$ with a prescribed error-term. If for all sufficiently large $\nu$ the error term is  $O(B^\nu)$  then one learns (by uniting the two interpretations) that  the $\nu$-th moment of the roots corresponding to $S$ is $O(B^\nu)$ for all sufficiently large $\nu$, and hence by the general philosophy of the method of moments, one can deduce that each root associated to $S$ is $O(B)$. One finally deduces that $|S|$ itself is $O(B)$. (In this application, one is of course thinking of $B = p^{n/2}$ for the $n$-dimensional case.) There are two general presentations of this style of argument: first,  Stepanov's method, exposed in \cite{Sch76} and \cite[\S 11.7]{IwaKow04}; see in particular Lemma 11.22 of \cite{IwaKow04} for the role of arbitrarily high moments in closing this argument. Second, there is the ``method of moments'' work of Hooley \cite{Hoo82}, whose consequences are summarized in \cite[Lemma 3.5]{Bro15}, and which also has a very clear, explicitly computed application in the work of \cite[Lemma 33]{BomBou09}.

 \section{Relation to the Cohen-Lenstra-Martinet moment predictions}\label{sec_CLM_predictions}

We start from the perspective of the heuristics of Cohen-Lenstra \cite{CohLen84} and Cohen-Martinet \cite{CohMar90}. We will show in this section that a collection of moment asymptotics predicted by these heuristics, and indeed even a collection of much weaker upper bounds, implies Conjecture \ref{conj_class}. We will furthermore note that another well-known conjecture formulated under the heuristics (the density statement) does not suffice alone to imply Conjecture \ref{conj_class}.

The Cohen-Lenstra and Cohen-Martinet heuristics  predict, for ``good'' primes $\ell$, the distribution of class groups and their $\ell$-torsion subgroups as $K$ varies over a family of number fields of fixed degree. For example,  the Cohen-Lenstra heuristics for the family $\Fscr_2^-(X)$ of imaginary quadratic fields imply a moment asymptotic for each odd prime $\ell$: namely that for every positive integer $k \geq 1$,
  \beq\label{CL_k_moment}
 \sum_{K \in \Fscr_2^-(X)} |\Cl_K[\ell]|^k \sim  c^-_{\ell,k} \sum_{K \in \Fscr_2^-(X)} 1,
 \eeq
 where $c^-_{\ell,k}$ is the number of $\Z/\ell\Z$ subspaces of the vector space $(\Z/\ell\Z)^k$ (see, e.g. \cite[Lemma 3.2]{Woo15}).  For real quadratic fields, the Cohen-Lenstra heuristics imply similar moment asymptotics with different values of the constant, say $c_{\ell,k}^+$.

Given a permutation group $G$,  for primes $\ell \ndiv |G|$, the Cohen-Martinet heuristics imply 
moment asymptotics analogous to (\ref{CL_k_moment}) for the family of fields with Galois group $G$ and fixed signature, again with different values of the constant playing the role of $c_{\ell,k}^-$.  (This follows from \cite[Theorem 6.1]{Wang2019}, after expressing  $ |\Cl_K[\ell]|^k$
as $|\Hom(\Cl_K,(\Z/\ell\Z)^k)|$ and then noting each homomorphism $\phi$ corresponds to a $\Gamma$ module homorphism from 
$\Cl_K$ to $\Ind_{1}^{\Gamma} \im \phi$.)

It is  worth mentioning the few known asymptotic results that verify cases of (\ref{CL_k_moment}) or its analogues for higher degree fields.
First, Davenport and Heilbronn \cite{DavHei71} proved (\ref{CL_k_moment}) for $\ell=3$ and $k=1$, with $c^-_{3,1} = 2$, as well as its real analogue with $c_{3,1}^+ = 4/3$. Second, Bhargava \cite{Bha05} proved the conjectured asymptotic for cubic fields and $2$-torsion, for the first moment ($k=1$). 
Although not originally included in the Cohen-Lenstra heuristics, Fouvry and Kl\"{u}ners \cite{FouKlu06} proved an analogous result related to  $4$-torsion when $K$ is quadratic, for all $k \geq 1$ (see Remark \ref{remark_FK}).

There is also an increasing body of recent work toward related asymptotics, though not exactly of the kind mentioned above. For example, recent work of Klys \cite{Kly16} gives certain asymptotic results on $3$-torsion in cyclic cubic fields; see also the recent work on 16-rank in quadratic fields, e.g. Milovic \cite{Mil17a},  Koymans and Milovic \cite{KoyMil16,KoyMil17a}, or work of Bhargava and Varma \cite{BhaVar15,BhaVar16} elaborating on \cite{DavHei71,Bha05}.  See also
results of Klys \cite{Kly16} on moments of $p$-torsion in cyclic degree $p$ fields (conditional on GRH for $p \geq 5$) and of Smith \cite{Smi16,Smi17x} on the distribution of the $2^{k}$-class groups in imaginary quadratic fields (recently extended by Koymans and Pagano \cite{KoyPag18x} to $\ell^k$-class groups of degree $\ell$ cyclic fields).
 
\subsection{Relation to Theorem \ref{thm_CLM_torsion}}
We will now show that if a moment asymptotic such as (\ref{CL_k_moment}) holds, as long as we assume information is known for  \emph{arbitrarily high} moments, we can deduce pointwise upper bounds on $|\Cl_K[\ell]|$ for each field $K$, with no exceptions.

In fact our assumption in Theorem \ref{thm_CLM_torsion} is weaker than (\ref{CL_k_moment}), taking the form of an upper bound: we assume the upper bound (\ref{CL_k_upper}) holds for one fixed $\al$ and all integers $k \geq 1$. For any $\al \geq 1$ the case of the inequality (\ref{CL_k_upper}) is  implied by (\ref{CL_k_moment}) or its appropriate analogue, and for all choices of $\al\geq 1$ the inequality is weaker than the conjectured Cohen-Lenstra-Martinet asymptotic for $\ell \ndiv |G|$.
For the cases of $\ell$ dividing $|G|$, for higher degree fields, in analogy with the quadratic case, we might also guess that for a given permutation group $G$, (\ref{CL_k_upper}) holds for $\ell $ dividing $|G|$  with $\alpha=1+\ep$, and that a proof of this in the case of $\ell$ dividing $ |G|$ is probably much more accessible than for good primes (see e.g. \cite{Kly16}).

\subsection{Proof of Theorem \ref{thm_CLM_torsion}}
 Recall the setting of Theorem \ref{thm_CLM_torsion}, which we introduced in \S \ref{sec_CL_intro}. We will use the fact that there exist  constants $C,B$ (not necessarily sharp) such that 
 \beq\label{FCB}
 |\Fscr(X)| \leq C X^B
 \eeq
  for all $X \geq 1$. 
  In fact this is known for each $n \geq 2$ if we consider the set of all degree $n$ extensions of $\Q$.
   By Malle's conjectures \cite{Mal02,Mal04}, it is expected that $B = 1$ should suffice, and this is known for $n=2$ by a classical computation, $n=3$ by \cite{DavHei71}, $n=4$ by \cite{CDO02}, \cite{Bha05}, and $n=5$ by \cite{Bha10a}.
For each $n \geq 6$ it is known that for the set of all degree $n$ extensions of $\Q$,  
\beq\label{fields}
|\Fscr(X)| \leq a_n X^{c_0 (\log n)^2},
\eeq
 for a constant $a_n$ depending only on $n$ and a certain absolute constant $c_0$ ($c_0=1.564$ suffices).
 This is due to Lemke Oliver and Thorne \cite{LemTho20}, and represents  the current state of a progression of works with ever-improving bounds: Couveignes \cite{Cou20}, Ellenberg and Venkatesh \cite{EllVen06}, Schmidt \cite{Sch95}.
 Of course for certain more restrictive collections of degree $n$ fields, we would expect the sharp value for $B$ to be even smaller, but we do not require this specificity here. We will note the dependence on $B$ in the initial steps of our  argument below, but later use that $B$ can be taken to depend only on the degree $n$.

 We now prove Theorem \ref{thm_CLM_torsion}. 
Let $n,\ell \geq 1$   be fixed, and assume that for a fixed $\al \geq 1$, (\ref{CL_k_upper}) is known for every integer $k \geq 1$.
Let $\Del>0$ be given, and then let $\Bscr_\Fscr(X,\Del)$ denote the ``bad'' set of fields $K \in \Fscr(X)$ such that 
$ |\Cl_K[\ell]| > D_K^\Del.$
Then (\ref{CL_k_upper}) implies that for every $X \geq 1$,
\begin{align*}
(X/2)^{\Del k}  \cdot  |\Bscr_\Fscr(X,\Del)  \intersect \{K  \in \Fscr(X): X/2< D_K \leq X\}|
&<  \sum_{K\in \bstack{\Bscr_\Fscr(X,\Del)}{X/2< D_K \leq X}} |\Cl_K[\ell]|^{k}  \\
 &\leq \sum_{\bstack{K \in \Fscr(X)}{X/2< D_K \leq X}} |\Cl_K[\ell]|^{k}  \ll_{n,\ell,k,\al}   |\Fscr(X)|^{\al}.
\end{align*}
We deduce that for each $k \geq 1$, for every $X \geq 1$,
\[ |\{K \in \Fscr(X): X/2< D_K \leq X, |\Cl_K[\ell]| > D_K^{\Del} \}|\ll_{n,\ell,k,\al,\Del}   X^{B\al  - \Del k}.\]
Applying this for some $k$ sufficiently large   that $B\al- \Del k < 0$, we conclude that for all $X$ sufficiently large (determined by $n,\ell,\al,\Del$),
\[ |\{K \in \Fscr(X): X/2< D_K \leq X, |\Cl_K[\ell]| > D_K^{\Del} \}| < 1\]
so that the set must be empty. Consequently, given any $\Del>0$, there exists $X_0=X_0(n,\ell,\al,\Del)$ such that for all fields $K \in \Fscr$ with $D_K \geq X_0$, we have $|\Cl_K[\ell]| \leq D_K^\Del$. Of course,  for all fields $K \in \Fscr$ with $D_K < X_0$ we can apply the trivial bound  
\[|\Cl_K[\ell]| \leq |\Cl_K| \ll_{\ep} D_K^{1/2+\ep} \ll_\ep X_0^{1/2 + \ep}  ,\]
for all $\ep>0$. Upon fixing $\ep=1/2,$  say,   this is a  bounded constant (depending on $n,\ell,\al,\Del$).
Combining these two results, we conclude that for the given $\Del$, for  all $K \in \Fscr$, 
\[ |\Cl_K[\ell]| \ll_{n,\ell,\al,\Del} D_K^{\Del} .\]
Since $\Del>0$ may be taken arbitrarily small in this argument, this concludes the proof of Theorem \ref{thm_CLM_torsion}.

\subsection{Further remarks}
\begin{remark}[Equivalent formulations]
This proof may clearly be adapted to functions $V : \Fscr \maps \R_{>0}$ other than $V(K) = |\Cl_K[\ell]|$, or in settings where $\Fscr$ is ordered by invariants other than the discriminant. Also, note that in the hypothesis of Theorem \ref{thm_CLM_torsion}, we have not ruled out the possibility that $\al$ is large. In fact, the statement that (\ref{CL_k_upper}) holds (for all $X \geq 1$) for a particular fixed $\al$  and an infinite sequence of arbitrarily large $k$ is equivalent to the statement that (\ref{CL_k_upper}) holds (for all $X \geq 1$)  for $\al = 1 + \ep_0$ with any $\ep_0>0$,  and all real $k \geq 1$. We provide details on this in Remark \ref{remark_1} in the appendix in \S \ref{sec_meas_space}.
\end{remark}

\begin{remark}[Density zero exceptions]\label{remark_moment_ae}
If one only knows (\ref{CL_k_upper}) for one value of $k$, one can still conclude that an upper bound holds for $|\Cl_K[\ell]|$, except for certain possible exceptional fields.  
 To be precise, let $n,\ell$ be fixed.  Suppose that for a single fixed value $k  \geq 1$, (\ref{CL_k_upper}) is known for every $\al>1$. Then  for each $\Del >0$, one can deduce that  $|\Cl_K[\ell]| \leq D_K^\Del$ holds for every field $K \in \Fscr$ except for at most a subset of density zero.
To see this, we argue as above to show that for any $\Del>0$ of our choice, for all $\al>1$, and for that fixed $k$,
\beq\label{moment_chain}
 |\Bscr_\Fscr(X,\Del) \intersect \{K  \in \Fscr(X): X/2< D_K \leq X\}| \ll_{n,\ell,k,\al, \Del} |\Fscr(X)|^{\al} X^{-\Del k} \quad \text{for all $X \geq 1$.}
 \eeq
To remove the dyadic restriction on the discriminant, we sum over 
 at most $\log_2 X + 1 \leq 2 \log_2 X$ dyadic ranges.  For the purposes of this discussion we suppose that $|\Fscr(X)|$ can be expressed as a power of $X$, and then we break the discussion into two cases:
 when $|\Fscr(X)|^{\al} X^{-\Del k}$ is  expressed as a power of $X$, the exponent is (i) non-negative; (ii)  negative.
In case (i), after summing over the dyadic ranges we can conclude that  
$|\Bscr_\Fscr(X,\Del)| \ll_{n,\ell,k,\al, \Del} |\Fscr(X)|^{\al} X^{-\Del k} \log X$ for all $X \geq 1$.
Consequently there exists $\al>1$ sufficiently close to $1$  that $|\Bscr_\Fscr(X,\Del)|/|\Fscr(X)| \ll |\Fscr(X)|^{\al-1}X^{ - \Del k} \log X$  has a negative exponent on the right-hand side. Then $\Bscr_\Fscr(X,\Del)$ has density zero in $\Fscr$.
In case (ii), then there exists a threshold $X_0 = X_0(n,\ell,k, \al,\Del)$ such that for $X \geq X_0$, 
the set on the left-hand side of (\ref{moment_chain}) is empty. Then there are only finitely many fields in $\Fscr$ with 
$|\Cl_K[\ell]| > D_K^\Del$, so that certainly these fields are density zero in $\Fscr$.
 \end{remark}

\begin{remark}[Illustrative case]\label{remark_FK}
We note a case in which moment bounds sufficient for the hypotheses of Theorem \ref{thm_CLM_torsion} have been proved: Fouvry and Kl\"uners have proved that for $\Fscr_2^\pm$ defined to be the family of real (respectively imaginary) quadratic fields with fundamental discriminants, for every integer $k \geq 1$,
\beq\label{FK_result}
\sum_{K \in \Fscr_2^\pm(X)} 2^{k \; \rk_2(2 \Cl_K)} \sim C^\pm_k |\Fscr_2^\pm(X)|
\eeq
as $X \maps \infty$, for a specific constant $C_k^\pm$. (We use additive notation for the finite abelian group $\Cl_K$, so that we write $2\Cl_K$ where Fouvry and Kl\"uners write $\Cl_K^2$; see \cite[Conj. 2]{FouKlu06} and Theorems 6--11 therein.) By Theorem \ref{thm_CLM_torsion}, (\ref{FK_result}) shows that $|(2\Cl_K)[2]| \ll_\ep D_K^\ep$  for all $\ep>0$. From this it follows that $|\Cl_K[4]| \ll_\ep D_K^\ep$ for all $K \in \Fscr_2^{\pm}(X)$. (Indeed, $ |\Cl_K[4]|=|2(\Cl_K[4])| \cdot |\Cl_K[4]/2(\Cl_K[4])|$. Within the first factor,  $|2(\Cl_K[4])|=|(2\Cl_K)[2]|$, which we have just bounded. Within the second factor, by the structure theorem of finite abelian groups,
 $|\Cl_K[4]/2(\Cl_K[4])|=|\Cl_K[2]|$, which is bounded by (\ref{Gauss_2}).)  Although this is an illustration of how Theorem \ref{thm_CLM_torsion} can be applied,  the result for $4$-torsion holds trivially for group-theoretic reasons once (\ref{Gauss_2}) is known; see \S \ref{sec_composite}. 
\end{remark}

\begin{remark}[Density statements]\label{remark_moments_density}
It is worth pointing out that  aside from moment asymptotics, there is another type of asymptotic, a density statement, which also follows from the Cohen-Lenstra and Cohen-Martinet heuristics.
For example, the density statement conjectured by the Cohen-Lenstra heuristics for imaginary quadratic fields  is that for every odd prime $p$ and finite abelian $p$-group $G$, 
\beq\label{CL_density}
\lim_{X\ra\infty} \frac{\sum_{\bstack{K\in \mathscr{F}^-_2(X)}{(\Cl_K)_p\isom G}} 1 }{\sum_{K\in \mathscr{F}^-_2(X)}1}= \frac{\prod_{i\geq 1} (1-p^{-i})}{|\Aut G|},
\eeq
where for any abelian group $A$ we denote by $A_p$  the Sylow $p$-subgroup of $A$. (For $p=2$, Smith \cite{Smi17x}  has proved that as $K$ varies over all imaginary quadratic fields, the distribution of $(2\Cl_K)_2$ obeys the heuristic expectation (\ref{CL_density}) for every finite abelian 2-group $G$; see also the extension \cite{KoyPag18x}.) 
In fact the Cohen-Lenstra and Cohen-Martinet conjectures make claims about averages over class groups of any ``reasonable'' real-valued function  of groups, where reasonable is not precisely defined, but is meant to include
both characteristic functions as in (\ref{CL_density}), and functions like $G\mapsto |G[p]|^k$ that are averaged in our moments, e.g. (\ref{CL_k_moment}). 

We have shown that moments such as \eqref{CL_k_moment} that are predicted by the Cohen-Lenstra-Martinet heuristics  imply Conjecture \ref{conj_class}.  In contrast, it is easy to see that density statements
such as \eqref{CL_density} do not formally imply Conjecture \ref{conj_class}, unless some additional information is used as input.  
For suppose we had a function $C$ from number fields to finite abelian groups such that for every odd prime $p$ and finite abelian $p$-group $G$, 
$$
\lim_{X\ra\infty} \frac{\sum_{\bstack{K\in \mathscr{F}(X)}{C(K)_p\isom G}} 1 }{\sum_{K\in \mathscr{F}(X)}1}= \frac{\prod_{i\geq 1} (1-p^{-i})}{|\Aut G|}.
$$
Then such averages would also hold if we replaced $C$ by any function $C'$ that only differs on a set of number fields that is of density zero in $\Fscr(X)$.  In particular, while 
the densities above still held for $C'$, we could impose $|C'(K)[\ell]|>D_K$ for infinitely many $K$.
\end{remark}

\section{Relation to the discriminant multiplicity conjecture}\label{sec_DMC}

Our next goal is to make precise the relationship of the discriminant multiplicity conjecture (Conjecture \ref{conj_DMC})  to the $\ell$-torsion conjecture  (Conjecture \ref{conj_class}).
\subsection{Context of the discriminant multiplicity conjecture}

 Recall the definition of Property $\Dbf_n(\varpi)$ (Property \ref{property_D}). The discriminant multiplicity conjecture posits that $\Dbf_n(0)$ should hold for each $n \geq 2$. 
(In the other direction, Ellenberg and Venkatesh \cite[p. 164]{EllVen05} have noted that genus theory shows that the number of degree extensions $K/\Q$ with a fixed Galois group and fixed discriminant $D$ can be as large as $D^{c/ \log \log D}$.)

For $n=3,4,5$ the best-known results are as follows. For cubic fields, $\Dbf_3(1/3)$ is known, by \cite[Cor. 3.7]{EllVen07}.
(That corollary in fact states that $|\Cl_K[3]| \ll D_K^{1/3+\ep}$ for all quadratic fields $K/\Q$; this implies that $\Dbf_3(1/3)$ by work of Hasse \cite{Has30}; see also Remark \ref{remark_Hasse}.)

For quartic fields, $\Dbf_4(1/2)$ holds. To see this, let us momentarily denote by $\Dbf_4(G;\varpi)$ the veracity of Property  \ref{property_D} restricted to those fields of degree $4$ and Galois group of the Galois closure isomorphic to $G$. Then $\Dbf_4(S_4;1/2)$ is due to \cite{Klu06b}, $\Dbf_4(A_4; 0.2784....)$ is due to \cite[Prop. 2.5]{PTBW20}, $\Dbf_4(K_4;0)$ follows from class field theory, as does $\Dbf_4(C_4;0)$  (see e.g. techniques in \cite{Wri89} also illustrated in \cite[Prop.  2.1]{PTBW20}), and finally $\Dbf_4(D_4;0)$ holds;   we include a brief proof of this last fact in an appendix in \S \ref{sec_app_D4}. 

For quintic fields, $\Dbf_5(199/200)$ is deduced in \cite[Prop. 2.4]{PTBW20} as an immediate consequence of a count for $S_5$-fields with a power-saving error term (in particular the version in \cite[Thm. 2.4]{EPW17}, which builds on \cite[Thm. 1]{ShaTsi14} and \cite[Thm. 1]{Bha10a}).

At present, for $n \geq 6$, nothing better than the bound trivially available from (\ref{fields}) is known for $\Dbf_n(\varpi)$.
For $n \geq 6$, if one further restricts to counting only those degree $n$ fields $K$ with $D_K=D$ and with Galois group of the Galois closure $ \simeq G$ for a certain fixed transitive subgroup $G \subseteq S_n$, then in some cases better results are known; see \cite[\S 2]{PTBW20}.

Conjecture \ref{conj_DMC} has been  known for some time to be closely connected to questions about $\ell$-torsion in class groups, for primes $\ell$. 
For example, Duke \cite{Duk95} showed that quartic fields of fixed discriminant $-q$ ($q$ prime) can be explicitly classified by odd octahedral Galois representations of conductor $q$, and the number of such fields can be expressed as in \cite{Hei71} as an appropriate average of the number of 2-torsion elements in the class groups of cubic number fields of discriminant $-q$. 
More generally, in \cite[p. 164]{EllVen05}, Ellenberg and Venkatesh stated that Conjecture \ref{conj_DMC} implies Conjecture \ref{conj_class} in the case of $\ell$ prime; these two ideas are also implicitly linked in the exposition of \cite{Duk98}.

 \subsection{Proof of Theorem \ref{thm_fields_cl}}
Let $K/\Q$ be a degree $n$ extension with Hilbert class field $H_K$, so that $\Gal(H_K/K) \simeq \Cl_K.$ Let $\ell$ be a fixed prime and write $\Cl_K$ additively, so that we have $\Cl_K [\ell] \simeq \Cl_K / \ell \Cl_K $ and  a corresponding extension $K \subset L \subset H_K$ given by the fixed field $L=H_K^{\ell \Cl_K}$.

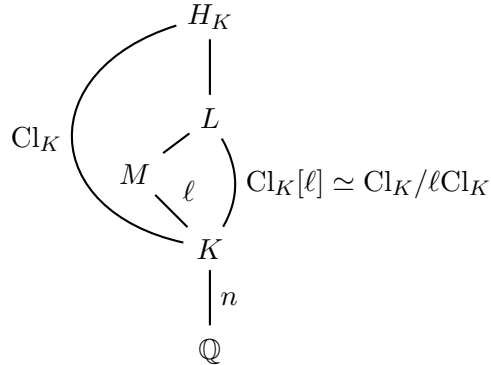
\begin{center}
\begin{tikzpicture}[-,>=stealth',shorten >=1pt,auto,node distance=2cm, thick,main node/.style]
\node[main node] (Kt) {$H_K$};
\node[main node] (L) [node distance=1.35cm, below of=Kt] {$L$};
\node[main node] (K) [node distance=1.75cm, below of=L] {$K$};
\node[main node] (Q) [node distance=1.35cm, below of=K] {$\Q$};
\node[main node] (M) [node distance=1cm, left of=K, above of=K] {$M$};

\draw[-](Kt) to [bend right = 75,min distance=2cm] node [swap]{$\Cl_K$} (K);
\draw[-] (Kt) to node {} (L);
\draw[-] (K) to node {$n$} (Q);
\draw[-, bend left] (L) to node {$\Cl_K[\ell] \simeq \Cl_K/\ell \Cl_K$} (K);
\draw[-] (L) to node{} (M);
\draw[-] (M) to node{$\ell$} (K);

\end{tikzpicture}
\end{center}

  Suppose $|\Cl_K[\ell]| = \ell^r$. Consider for each surjection $\phi : \Cl_K[\ell] \maps \Z/\ell\Z$ the corresponding extension $K \subset M \subset L$ with $M/K$ of degree $\ell$. 
There are $N(\ell,r):=(\ell^r-1)/(\ell-1)$ such extensions $M/K$, each of which is degree $n\ell$ over $\Q$ and naturally satisfies 
\[\Disc M/\Q = ( \Nm_{K/\Q} \Disc M/K) \cdot (\Disc K/\Q)^\ell.\]
But since $M/K$ lives inside the unramified extension $H_K/K$, the first factor on the right-hand side is 1, so that $D_M = D_K^\ell$. 
Now supposing we have assumed Property $\Dbf_{n\ell}(\Del)$ holds, we know that for every $\ep>0$ there exists a constant $C_{n\ell,\ep}$ such that the number of degree $n\ell$ extensions of $\Q$ with discriminant precisely $D_K^\ell$ is $\leq C_{n\ell,\ep} (D_K^{\ell})^{\Del+\ep}$.
Then we can conclude that 
\[ |\Cl_K[\ell]| = (\ell-1) N(\ell,r) + 1 \leq (\ell-1) C_{n\ell,\ep} (D_K^{\ell})^{\Del+\ep}+1. \]
Thus for this fixed pair $n,\ell$, Property $\Cbf_{n,\ell}(\Del')$ holds with $\Del' = \ell \Del$, proving our claim.

  \begin{remark}\label{remark_Hasse}
  For $\ell$ an odd prime, there is a partial converse to Theorem \ref{thm_fields_cl}. Let $\Dbf_\ell(D_\ell;\Del)$ denote the veracity of Property \ref{property_D} for those degree $\ell$ fields with Galois group of the Galois closure isomorphic to $D_{\ell} \subset S_\ell$, the dihedral group of $2\ell$ elements (symmetries of a regular $\ell$-gon). Then for every odd prime $\ell$, for every $\Del>0$,
$ \Dbf_\ell (D_{\ell}; \Del)$ holds if $\Cbf_{2,\ell}(\Del)$ holds.
  For $\ell=3$ this is the conclusion of Hasse \cite{Has30}, and in particular shows that the upper bound for $3$-torsion in quadratic fields controls the number of cubic extensions of $\Q$. In general, for this deduction see Kl\"uners \cite[Lemma 2.3]{Klu06}, which applies to degree $\ell$ $D_\ell$-extensions of any fixed number field. Recently \cite[Thm. 1.1]{CohTho16x} and  \cite[Cor. 1.6]{FreWid18x} have used this relation to capitalize on the recent improvements of \cite{EPW17,PTBW20} on $\ell$-torsion in quadratic fields in order to bound from above the number of degree $\ell$ $D_\ell$-fields of bounded discriminant. 
   \end{remark}
  
\section{Relation to a generalized Malle conjecture}\label{sec_gen_Malle}
In this section we recall the formal statement of a generalized Malle conjecture as in Ellenberg and Venkatesh \cite{EllVen05}, and state how via the method of moments it implies the discriminant multiplicity conjecture (Conjecture  \ref{conj_DMC}).
Counting number fields of bounded discriminant is an average version of the pointwise bound conjectured in the discriminant multiplicity conjecture. Of course, if $\Dbf_n(\varpi)$ is known then we may immediately deduce an upper bound for the average. At first glance, it is surprising that an implication in the other direction also exists; this would seem to contradict the philosophy of \S \ref{sec_moments}, which relied on bounding arbitrarily high moments, and not just an average, in order to deduce that a pointwise bound holds at \emph{all} points in a discrete space. The strategy of \cite{EllVen05}  is that by converting $k$-th moments into averages of $k$-products, one can deduce pointwise bounds only under assumptions on averages.

\subsection{Context of the Malle conjecture}
Recall that it is conjectured that for every $n \geq 2$, if we define $Z_n(\Q;X)$ to be the set of fields $K/\Q$ of degree $n$ (within a fixed closure $\overline{\Q}$), then 
\[ |Z_n(\Q;X)| \sim c_n X \]
for a particular constant $c_n$ (this is recorded for example as a ``folk conjecture'' in \cite[Conj. 1.1]{EllVen05}). This is easy to show for $n=2$ (in which case one is effectively counting square-free numbers, see e.g. the appendix in \cite{EPW17}). It has been proved in deep work of Davenport-Heilbronn \cite{DavHei71} for $n=3$; Cohen, Diaz y Diaz, and Olivier \cite{CDO02} and Bhargava \cite{Bha05} for $n=4$ (the $D_4$ and non-$D_4$ cases, respectively); and Bhargava \cite{Bha10a} for $n=5$. For any degree $n \geq 6$, such an asymptotic, or even an upper bound better than (\ref{fields}), is not known. (Lower bounds are also difficult; \cite{BSW16x} holds the current record $\gg X^{1/2 + 1/n}$.)

If we further define $Z_n(\Q;G,X)$ to restrict to those fields in $Z_n(\Q;X)$ with $\Gal(\tilde{K}/K) \simeq G$, then Malle \cite{Mal02}, \cite{Mal04} has presented heuristics for upper and lower bounds 
\beq\label{Malle}
 X^{a_G} \ll |Z_n(\Q;G,X)| \ll X^{a_G+\ep}
 \eeq
 as $X \maps \infty$ with $a_G^{-1} =\min \{ \mathrm{ind}(g) : g \in G, g \neq \mathrm{id}\}$. (Here the index of an element $g \in S_n$ is defined as $\mathrm{ind}(g) = n-j$ where $j$ is the number of orbits of $g$ acting on a set of $n$ elements.) These ideas can also be adapted, at least heuristically according to the so-called Malle-Bhargava principle \cite[\S 10]{Woo16}, to take into account finitely  or infinitely many local conditions. Proving that this principle holds in appropriate situations is currently of central interest in number theory.

\subsection{Ellenberg and Venkatesh's moment argument}
We first recall, somewhat informally, how a standard method of moments argument would approach the discriminant multiplicity conjecture. Let us consider (\ref{moment1}) and (\ref{moment2}), adapted to the setting where for each $m \in \mathbb{N}$, $f(m)$ denotes the number of fields  $K \in Z_n(\Q;G,X)$ such that $D_K = m$. Then after adapting (\ref{moment1})  to this setting, as long as we can prove that for a sequence of arbitrarily large $k$ we have for all $X \geq 1$ that 
\beq\label{moment3}
\sum_{m \leq X} |f(m)|^k  \leq X^\be 
\eeq
for some fixed $\be$, we will be able to conclude similarly to (\ref{moment2}) that for all $X \geq 1$, 
 \[ | \{ X/2 < m \leq X : f(m) > m^\al \}| \ll  X^{\be - \al k} .\]
Then no matter the fixed value of $\be$, and for any $\al>0$ arbitrarily small, by applying this with sufficiently large $k$, we could deduce that for sufficiently large $X$ the set on the left-hand side is empty. From this we could deduce that there are at most $O_\al(m^\al)$ fields $K \in \Z_n(\Q;G,X)$ with $D_K=m$, as desired. 
A significant obstacle to this approach is that in general, bounding even the first moment
 \[ |Z_n(\Q;G,X)| = \sum_{m \leq X} f(m) \]
 remains completely out of reach at present (e.g. for $G\simeq S_n$ with $n \geq 6$), so that proving (\ref{moment3}) may seem an unrealistic starting point. 

This is where the ideas of Ellenberg and Venkatesh \cite[Remark 4.7, Prop. 4.8]{EllVen05} play a role, by cleverly translating the $k$-th moment estimate (\ref{moment3}) into merely an average estimate, but for a different set of fields than $Z_n(\Q;G,X)$.
Consequently, they show that a generalized version  of the \emph{average} upper bound in (\ref{Malle}) would imply Conjecture \ref{conj_DMC}.

The formulation of Ellenberg and Venkatesh is in terms of a field invariant called an $f$-discriminant, with $f$ a rational class function (not to be confused with the function $f$ used above). Precisely, if $G$ is a fixed transitive subgroup of $S_n$, then a map $f$ from the set $\mathcal{C}_G$ of nontrivial conjugacy classes of $G$ to $\Z_{\geq 0}$ is said to be a rational class function if it is invariant under the $\Gal (\overline{\Q}/\Q)$-action 
 and has the property that for $g \in G$, $f(g)=0$ if and only if $g = \{\mathrm{id}\}$. 
One can define a corresponding $f$-discriminant of any Galois extension $L/\Q$ with Galois group $G$. Precisely, for any $p \ndiv |G|$ (so that certainly $p$ is at most tamely ramified in $L$ and the inertia group at $p$ is cyclic), let $c_p \in \mathcal{C}_G$ be the image of a generator of tame inertia at $p$.  
 Then the $f$-discriminant $D_f(L)$ may be defined by 
\[ D_f(L) = \prod_{p \ndiv |G|} p^{f(c_p)}.\]
In particular, if $f(g) = \mathrm{ind}(g)$, then $D_f(L)$ is the prime-to-$|G|$ part of the standard absolute discriminant of any of the $n$ conjugate degree $n$ extensions $L_0/\Q$ with Galois closure $L$, corresponding via Galois theory to the element stabilizers in $G$.

Ellenberg and Venkatesh ask whether for such $f$-discriminants, an appropriate analogue to (\ref{Malle}) holds. Given $G$, Ellenberg and Venkatesh define for a rational class function $f$ the invariant
\[ a_G(f) = \max_{c \in \mathcal{C}_G, c \neq \{ \mathrm{id}\}} f(c)^{-1}.\]
They then propose \cite[Ques. 4.3, Remark 4.7]{EllVen05} that if $Z^{(f)}(\Q,G;X)$ counts Galois extensions of $\Q$ with Galois group $G$ and $f$-discriminant at most $X$, then for every $\ep>0$,
\beq\label{Malle_EV}
|Z^{(f)}(\Q,G;X)| \ll_{G,f,\ep} X^{a_G(f) + \ep}.
\eeq
(In fact they propose a more refined asymptotic, but this upper bound version suffices for our considerations.)

Under these definitions, we can see that (\ref{Malle_EV}) suffices to imply Conjecture \ref{conj_DMC}, and hence Conjecture \ref{conj_class}.
  \begin{letterthm}[{\cite[Prop. 4.8]{EllVen05}}]\label{thm_EV_count}
Given any group $G$, define on $G^k = G \times G \times \cdots \times G$ the rational class function that is identically 1  away from the identity, that is $f(c_1,c_2,\ldots, c_k) = 1$ for any choice of conjugacy classes $c_1,c_2,\ldots, c_k$ in $\mathcal{C}_G$  that are not all trivial. 
If a generalized Malle conjecture (\ref{Malle_EV}) holds for all $G$, for this choice of $f$, then Conjecture \ref{conj_DMC} holds.
 \end{letterthm}

We merely sketch the ideas, directly following \cite[Prop. 4.8]{EllVen05}. Since isomorphism classes of degree $n$ $G$-fields are in finite-to-one correspondence (that is, $m_G$-to-one, where $m_G$ depends only on $G$)
 with isomorphism classes of degree $|G|$ Galois $G$-fields, to verify  Conjecture \ref{conj_DMC} it suffices to prove that for any fixed Galois group $G$, the number of Galois extensions $K/\Q$ with discriminant $D_K = D$ is $\ll_\ep D^\ep$ for every $\ep$. The proof hinges on the fact that for $f$ as chosen in the theorem statement, the $f$-discriminant of a field $K$ is of the shape 
\[ \prod_{p|D_K,\;  p \ndiv |G|} p,
\]
 in which the power of the prime is always 1. 
Moreover, for the particular choice of $f$ in Theorem \ref{thm_EV_count}, we have $a_{G^k}(f)=1$ for all $k$; this fact will play a key role.

For each fixed integer $D \geq 1$ let $m(D)$ denote the number of Galois fields $K_1,\ldots,K_{m(D)}$, each with Galois group $\simeq G$ and each with $D_{K_i} = D$; our goal is to show that $m(D) \ll D^\ep$. For any fixed $k \leq m(D)$ we can generate on the order of ${m(D) \choose k} \gg_k m(D)^k$ composite fields $L=K_{j_1}\cdots K_{j_k}$ each  with Galois group $\simeq G^k$, and each with $f$-discriminant $D_f(L/\Q) \leq D$, with $f$ as defined as above.

Now suppose that there exists some $\al>0$ such that for a sequence of arbitrarily large $D$, we have  $m(D) \geq D^{\al}$. Then for a sequence of arbitrarily large $D$ we would obtain at least $ \gg_k D^{\al k }$ elements in $Z^{(f)}(\Q,G^k;D)$. Comparing this to the conjectured statement (\ref{Malle_EV}) in the case of the group $G^k$, we see that we must have 
$\al k \leq a_{G^k}(f)$, but by construction, we have $a_{G^k}(f)=1$ for all $k$.
 Since $k$ could be taken arbitrarily large, we would conclude that $\al$ must be arbitrarily small, implying Property $\Dbf_n(0)$ must hold.

\section{Counting elliptic curves}\label{sec_EC}

\subsection{Context of counting elliptic curves}
Given a positive integer $q$, let $E(q)$ denote the number of isogeny classes of elliptic curves over $\Q$ with conductor $q$.
Let us clarify the relation between $E(q)$ and the number $E'(q)$ of isomorphism classes of elliptic curves over $\Q$ with conductor $q$. As used in \cite[p. 12]{DukKow00}, given any elliptic curve $E/\Q$, at most 8 $\Q$-isomorphism classes are $\Q$-isogenous to $E$; thus $E(q) \leq E'(q) \leq 8 E(q)$ (see \cite[Ch.IX, Example 6.4]{Sil09} for this fact).
Thus an analogous conjecture holds for $E'(q)$, and we  may proceed by considering either $E(q)$ or   $E'(q)$; in particular the statement of our results will hold for both $E(q)$ and $E'(q)$, up to a constant factor.

As we recalled in Conjecture \ref{conj_EC},  Brumer and Silverman \cite[Conj. 3]{BruSil96} conjectured  that for every $q$, 
$E(q) \ll_\ep q^{\ep} $
for all $\ep>0$.
Indeed,  Brumer and Silverman make a  more precise conjecture: that there exist absolute constants $\kappa, \kappa'$ such that for any finite set $S$ of primes, the number of elliptic curves with good reduction outside of $S$ is at most $\kappa M_S^{\kappa'/\log \log M_S}$, where $M_S = \prod_{p \in S}p$.
\begin{remark} Correspondingly, Brumer and Silverman's original formulation of Conjecture \ref{conj_class} is the more precise question of whether 
 for every field $K$ of degree $n$ one has 
\[ \log_\ell |\Cl_K[\ell]| \ll_{n,\ell} \log D_K / \log \log D_K.\]
See also \cite[Conj. 3.3.A, Conj. 3.3.B, Thm. 7.5]{BelLub17} for other possible specific rates of growth, and their relevance for counting the number of conjugacy classes of non-uniform   lattices in semisimple Lie groups.
\end{remark}

Upper bounds for $E(q)$ follow from bounding $3$-torsion in quadratic fields.
Define for every $X \geq 1$ the quantity
\[ H_3(X)  = \max_{K \in \Fscr_2(X)} |\Cl_K[3]|,\]
in which $\Fscr_2(X)$ denotes the family of quadratic fields (real or imaginary) with absolute discriminant at most $X$ (in absolute value). 
Then Brumer and Silverman showed in \cite[p. 99-100]{BruSil96} that for every $q$, 
\beq\label{BruSil_ptwise}
 E(q) \ll_\ep q^\ep H_3(1728q),
 \eeq
 for every $\ep>0$.
 This was then sharpened in the proof of  \cite[Thm. 4.5]{HelVen06} to 
 \beq\label{HelVen_ptwise}
 E(q) \ll_\ep q^\ep (H_3(1728q))^{2 \be/ \log 3},
 \eeq
 for every $\ep>0$, where $\be =  0.2782...$ and hence $2 \be/ \log 3 = 0.5064...$, so that $\Cbf_{2,3}(\Del)$ implies $E(q) \ll_\ep q^{(0.5064...)\Del+\ep}$. 
Originally the only bound available to Brumer and Silverman was the trivial bound $H_3(X) \ll X^{1/2+\ep}$ via (\ref{trivial_bound}); now that $\Cbf_{2,3}(1/3)$ is known by \cite{EllVen07}, it follows that 
\[ E(q) \ll_\ep q^{2 \be/(3 \log 3)+\ep} \ll_{\ep} q^{0.1688...+\ep} \]
for all $q$. This is currently the best result known toward Conjecture \ref{conj_EC}.

In addition, Brumer and Silverman have proved that Conjecture \ref{conj_EC} is true under the assumption of the Generalized Riemann Hypothesis and a weak version of the Birch-Swinnerton--Dyer Conjecture (\cite[Thm. 4]{BruSil96}). This is an interesting counterpart to the situation for bounding $\ell$-torsion, for which it appears that GRH implies partial results, but stops far short of proving Conjecture \ref{conj_class} in full (see \S \ref{sec_GRH}).

 From Theorem \ref{thm_fields_cl} and (\ref{HelVen_ptwise}), we make a simple observation, recalling the notation of Property \ref{property_D}.
\begin{thm}\label{thm_EC_6}
If Property $\Dbf_6(\Del)$ holds, then 
\[ E(q) \ll_{\ep} q^{6\be\Del/\log 3+\ep} \ll_\ep q^{(1.5193 \ldots)\Del + \ep}
\]
 holds for all $q \geq 1$. In particular, if it is known that for every $D$ there are at most $\ll_\ep D^{\ep}$ sextic fields $K/\Q$ of discriminant $D$, then Conjecture \ref{conj_EC} holds.
\end{thm}

\subsection{Average results}
One can also ask about an average   or asymptotic conjecture related to  Conjecture \ref{conj_EC}. Here we distinguish briefly between ordering elliptic curves by discriminant or by conductor. 
We first consider  ordering by discriminant. 
Brumer and McGuinness \cite[\S 5]{BruMcg90}  have conjectured that the number of elliptic curves over $\Q$ whose minimal discriminant has absolute value less than $X$ is asymptotically of size $cX^{5/6}$ for an explicitly predicted constant $c$ (see also Watkins \cite[Conj. 2.1]{Wat08}).
Each elliptic curve over $\Q$ in its Weierstrass form $y^2 = x^3-27c_4x - 54c_6$ with integers $c_4, c_6$ has discriminant $\Del$ with 
\beq\label{E_disc}
1728\Del = c_4^3 - c_6^2.
\eeq
(Here the discriminant $\Del$ and $c_4,c_6$ are the standard invariants as in \cite[p. 42 Ch. III \S 1]{Sil09}.)
   Very roughly speaking, we may think of the numerology as suggesting that for $\Del$ fixed, most of the solutions to the equation (\ref{E_disc}) must come from $c_4 \ll \Del^{1/3}$ and $c_6 \ll \Del^{1/2}$. 
Counting elliptic curves ordered by discriminant, Fouvry, Nair and Tenenbaum   \cite{FNT92} proved the weaker upper bound $\ll X^{1+\ep}$ for every $\ep>0$ and a lower bound of the order of magnitude $\gg X^{5/6}$.
Any upper bound that is $o(X)$ would imply that the integers that are discriminants of elliptic curves are zero-density among all integers (a difficult open problem; see \cite{You15}). 

We next consider ordering by conductor $q$.
Recall  that the same primes divide the discriminant and the conductor, and the conductor is a divisor of the discriminant; the precise relationship between the two is related to Szpiro's conjecture \cite{Szp90} (for example predicting that $\Del \ll_\ep q^{6+\ep}$). 
 Watkins \cite[Heuristic 4.1]{Wat08} adapted the argument of Brumer-McGuinness \cite[\S 5]{BruMcg90}  by analyzing heuristically how often a large power of a prime divides the discriminant, and proposed the following expectation:
\begin{conjecture}\label{conj_EC_avg}
Let $E'(q)$ denote the number of isomorphism classes of elliptic curves of conductor $q$. There exists a constant $c'>0$ such that 
\[\sum_{1 \leq q \leq Q} E'(q) \sim c' Q^{5/6}\]
as $Q \maps \infty$.
\end{conjecture}
 Duke and Kowalski \cite[Prop. 1]{DukKow00} have proved the weaker result that for $Q \geq 1$, 
\beq\label{DK_average_result}
 \sum_{1 \leq q \leq Q} E(q) \ll_\ep Q^{1+\ep}.
 \eeq
See also \cite{Won01a,Won01b}, which further proves a lower bound $\gg_\ep Q^{5/6-\ep}$. 
All these works \cite{DukKow00,FNT92,Won01a} argue by applying the asymptotic of Davenport and Heilbronn \cite{DavHei71} for the average of $|\Cl_K[3]|$.  
\begin{remark}\label{remark_EC}
It is worth noting that an argument truly is needed to deduce (\ref{DK_average_result}) or our theorem on moments. Armed only with (\ref{BruSil_ptwise}), the statement we can trivially deduce is that
\[ \sum_{q \leq Q} E(q)^k \ll_\ep Q^\ep \sum_{q \leq Q} H_3(1728 q)^k \ll_\ep  Q^{1+\ep} H_3(1728 Q)^k. \]
Since $H_3$ is a sup-norm, even one violation of a pointwise upper bound $|\Cl_K[3]| \leq D_K^\Del$ would prevent us from concluding that 
 the expression above is bounded by $\ll_\ep Q^{1+k\Del +\ep}.$ Similar considerations apply if we use (\ref{HelVen_ptwise}).
\end{remark}

\subsection{Moment results: proof of Theorem \ref{thm_EC}}

The proof of Theorem \ref{thm_EC} builds on the arguments in   \cite{BruSil96}, \cite{DukKow00} and \cite{HelVen06}; we begin by recalling the set-up of \cite{BruSil96}. An elliptic curve of conductor $q$ has good reduction outside of the primes dividing $q$, and so  it suffices to count the number $E_g'(q)$ of elliptic curves with good reduction outside of the set $S_q = \{p|q, \text{$p$ prime}\} \union \{2,3\}$, and for any fixed $\kappa\geq 1$ we write
\[ \sum_{q \leq Q} E'(q)^\kappa  \leq \sum_{q \leq Q} E_g'(q)^\kappa. \]
Let us fix $q$ and the corresponding set $S_q$. Then any elliptic curve $E/\Q$ counted by $E_g'(q)$ has 
\beq\label{disc_ad}
1728\Del_E = ad^6
\eeq
 for integers $a,d$ with $a$ being $6$-th power free, and such that $(c_6(E)/d^3,c_4(E)/d^2)$ is an $S_q$-integral point on the elliptic curve $\Escr_a$ defined by 
\[ \Escr_a : Y^3 = X^2 + a.\]
 Let us denote by $A(q)$ the set of values $a$ that arise in this manner as $E$ ranges over $E_g'(q)$. Furthermore, given any $S_q$-integral point $P$ on one such $\Escr_a$, let us denote by $C(P)$ the set of elliptic curves $E$ counted by $E_g'(q)$ that give rise to the point $P$ by following the procedure above. 
Brumer and Silverman's argument (see also \cite{DukKow00}) shows that 
\[  \sum_{q \leq Q} E_g'(q)^\kappa \leq \sum_{q \leq Q}  ( \sum_{a \in A(q)} \sum_{P} |C(P)| )^\kappa, \]
in which $P$ varies over the $S_q$-integral points on $\Escr_a$. For each $P$, Brumer and Silverman show  that $|C(P)| \leq 2^{\om(q)+5} \ll_\ep q^\ep$ for every $\ep>0$ (see \cite[p. 100]{BruSil96}, as usual recalling that $\om(q) \ll \log q/ \log \log q$). Moreover, they show  that for each such $a,q$, the number of $S_q$-integral points on $\Escr_a$ is 
$
\ll_\ep q^\ep |\Cl_{K_a}[3]|
$
 for every $\ep>0$, where $K_a$ is the quadratic field $\Q(\sqrt{-a})$. (See the line leading to \cite[Eqn. (4)]{BruSil96}; in their equation we note that $\# S = \#S_q \leq \omega(q) +2$, so that for any constant $c$ we have $c^{\#S} \ll  q^\ep$ for all $\ep>0$.) 
More recently, due to improved counts for integral points on a fixed Mordell elliptic curve (such as $\Escr_a$), we can improve this by \cite[Thm. 4.5]{HelVen06} to write that the number of $S_q$-integral points on $\Escr_a$ is 
\[ \ll_\ep a^\ep \exp\{ 2 \log_3 |\Cl_{K_a}[3]|(\be + \ep)\} \ll_\ep a^\ep |\Cl_{K_a}[3]|^{2\be / \log 3}, \]
for any $\ep>0$,
where $\be = 0.2782...$ is a numerical constant defined explicitly as $\be(0)$ in \cite[Thm. 3.8]{HelVen06}.
We note that $2\be/ \log 3 = 0.5064...$ is strictly smaller than 1, so that this is advantageous.

Thus we may conclude that for every $\ep>0$,
\[  \sum_{q \leq Q} E_g'(q)^\kappa \ll_\ep  Q^\ep \sum_{q \leq Q}  ( \sum_{a \in A(q)}a^\ep |\Cl_{K_a}[3]|^{2\be / \log 3} )^\kappa. \]
Now, as in the argument in \cite[\S 3]{DukKow00}, for each fixed $q$ we will pass from a sum over $a \in A(q)$ to a sum over the set of (distinct) corresponding square-free kernels $a'$; we set $A'(q) =  \{ a' : a \in A(q)\}$. For any elliptic curve $E \in E_g'(q)$   associated to $a$, so that the discriminant relation (\ref{disc_ad}) holds,  we see from (\ref{disc_ad}) and the fact that the corresponding  square-free kernel $a'$ is \emph{square-free} that  $|a'| \leq 1728q$ and $a' | \del q$, where 
\[\del = 6/\gcd(6,q).\]
In general, denoting by $s(a) = a'$ the square-free kernel of an integer, we have for any $a'$ that 
\[ 
\# \{ a \leq X : s(a) = a'\} \ll_\ep X^\ep 
\]
for all $X \geq 1$ and for all $\ep>0$;
see e.g. \cite[Lemma 4.4]{HBP17}. 
Since $a$ is $6$-th power free, we may apply this in our setting with $X\leq 2^5 3^5 q^5$, say. Thus there are at most $\ll_\ep q^\ep$ such values $a$ corresponding to each $a' \in A'(q)$.

  Thus we have shown that
\[
  \sum_{q \leq Q} E_g'(q)^\kappa \ll_\ep  Q^\ep \sum_{q \leq Q}  ( \sum_{a'|\del q}|\Cl_{K_{a'}}[3]|^{2\be / \log 3}  )^\kappa, 
\]
where we have used the fact that $a^\ep \ll q^\ep$ for any $\ep>0$.
On the right-hand side, we initially only consider square-free $a'$, but we can enlarge the right-hand side by considering any such divisors $a'$. 
Applying H\"older's inequality to the right-hand side, this becomes
\[
 \sum_{q \leq Q} E_g'(q)^\kappa \ll_\ep Q^\ep \sum_{q \leq Q} (d(\del q))^{\kappa-1}\sum_{a'|\del q}|\Cl_{K_{a'}}[3]|^{2\be \kappa / \log 3} \ll_\ep Q^{2\ep} \sum_{q \leq Q} \sum_{a' |\del q} |\Cl_{K_{a'}}[3]|^{2\be \kappa / \log 3},
\]
since for all integers $m$, the divisor function satisfies $d(m) \ll_\ep m^\ep$ for all $\ep>0$. 
Note that we may trivially enlarge the right-most quantity to
\beq\label{partial_sum_app}
\ll_\ep Q^{2\ep} \sum_{q \leq 6Q} \sum_{a' |q} |\Cl_{K_{a'}}[3]|^{2\be \kappa / \log 3} .
 \eeq
We apply the following elementary lemma. 
\begin{lemma}
For any sequence $\al_{n}$ of non-negative real numbers, define $A_t   = \sum_{n \leq t} \al_n$. Suppose it is known  that for all $t \geq 1$ we have
$
A_t \leq t^\Theta,
$
for some $\Theta \geq 1$.
Then for all $Q \geq 1$,
\[  \sum_{ q \leq Q} \sum_{n|q} \al_n  \ll_\Theta Q^{\Theta} \Lscr(\Theta) \]
in which $\Lscr(\Theta)= \log Q$ if $\Theta=1$ and $\Lscr(\Theta)=1$ if $\Theta>1$. 
\end{lemma}
\begin{proof}First note that (using positivity)
\[
  \sum_{ q \leq Q} \sum_{n|q} \al_n = \sum_{n \leq Q} \al_n \sum_{\bstack{q \leq Q}{n|q}} 1 =\sum_{n \leq Q} \al_n \lfloor Q/n \rfloor 
 \leq  Q \sum_{n \leq Q} \frac{\al_n}{n}.
\]
 Then by partial summation and the assumption on $A_t$ we have 
\[ \sum_{n \leq Q} \frac{\al_n}{n} = A_Q Q^{-1} + \int_1^Q A_t t^{-2} dt  \leq A_Q Q^{-1}  + \int_1^{Q} t^{\Theta-2} dt
\ll_\Theta Q^{\Theta-1} \Lscr(\Theta),
\]
which suffices to prove the lemma.
\end{proof}

We apply the lemma in (\ref{partial_sum_app}) with $\al_n  = |\Cl_{K_{n}}[3]|^{2\be \kappa / \log 3}$ and $\Theta = \Theta_k$ the exponent assumed to hold in (\ref{assume_k_moment_Cl3}) for the value $k = 2\be \kappa / \log 3$. Note that for any $k>0$,  $\Theta_k \geq 1$ since $|\Fscr_2(X)| \gg X$ and $|\Cl_K[3]| \geq 1$ for all fields $K$. 
We may conclude from (\ref{partial_sum_app}) that 
\[ \sum_{q \leq Q} E_g'(q)^\kappa \ll_{k,\ep}  Q^{2\ep} Q^{\Theta_k} \Lscr(\Theta) \ll_{k,\ep} Q^{\Theta_k +2 \ep}\]
for all $\ep>0$.
This suffices to prove the theorem, upon noting that 
\[ \kappa = \frac{k \log 3}{2\be} =k \cdot \ga= k \cdot 1.9745... \]


Corollary \ref{cor_EC} follows from Theorem \ref{thm_EC} immediately, upon applying the results of \cite[Cor. 1.4]{HBP17}. To situate the strength of Corollary \ref{cor_EC}, we briefly observe that for any $k$ such that $\ga k \geq 1$, by combining 
(\ref{DK_average_result}) with the pointwise bound $E(q) \ll_\ep q^{1/(3\ga) + \ep}$ in (\ref{Eq_EV}), 
\beq\label{trivial}
 \sum_{q \leq Q} E(q)^{k\ga} \leq \max_{q \leq Q} E(q)^{k\ga - 1} \sum_{q \leq Q} E(q) \ll Q^{1 + (k\ga-1)/(3\ga) + \ep}. 
 \eeq
For example, in the case $k=2/\ga \geq 1$ the exponent (ignoring $\ep$) in (\ref{trivial}) is $1+1/(3\ga) = 1.1688...$ which is larger than the exponent $1+(2/\ga-1)/3 =1.0043 ...$ obtained by applying Theorem \ref{thm_EC}, and applying $|\Cl_K[3]|\ll D_K^{1/3+\ep}$ to one factor and using Davenport-Heilbronn to bound the remaining average. 

As another example, in the case $k=2$, (\ref{trivial}) gives the exponent $5/3-1/(3\ga) = 1.4978...$ while Corollary \ref{cor_EC} gives 
the exponent $23/18=1.2777...$.

\section{Known results toward Conjecture \ref{conj_class}, and relation to GRH}\label{sec_previous}

We have focused thus far on heuristics to support Conjecture \ref{conj_class}. To conclude, we provide an overview of 
recent progress toward the conjecture, focusing first on results that hold for \emph{all} fields of a given degree.  Then we focus on a relationship to the Generalized Riemann Hypothesis.
We will use the notation $\Cbf_{n,\ell}(\Del)$ defined in Property \ref{property_C}.

\subsection{Composite $\ell$}\label{sec_composite}
First we remark that  while we have stated Conjecture \ref{conj_class} only for primes $\ell$, we expect it to hold for all integers $\ell \geq 1$. Indeed, let $K/\Q$ be a number field. Since $\Cl_K$ is a finite abelian group, it is isomorphic to
\[ 
C_{\ell_1^{a_{1,1}}} \times \cdots \times C_{\ell_1^{a_{1,r_1}}} \times \cdots \times C_{\ell_s^{a_{s,1}}} \times \cdots \times C_{\ell_s^{a_{s,r_s}}},
\]
for distinct primes $\ell_1,\ldots, \ell_s$,
and cyclic groups $C_{\ell^a} \isom \Z/ \ell^a \Z$. For any prime $\ell | \, |\Cl_K|$, letting $r$ be the number of cyclic groups in this decomposition corresponding to $\ell$,  then $|\Cl_K[\ell]| = \ell^r$. (If $\ell \ndiv |\Cl_K|$ then this remains true, with $r=0$ and $|\Cl_K[\ell]|=1$.)

 Furthermore, by the above decomposition,  $|\Cl_K[\ell]|$ is multiplicative as a function of $\ell$, that is, for integers $(\ell_1,\ell_2)=1$,
$
   |\Cl_K[\ell_1\ell_2]|  = |\Cl_K[\ell_1]| \cdot |\Cl_K[\ell_2]|.
$
 Consequently if $(\ell_1,\ell_2) =1$ and $\Cbf_{n,\ell_1}(\varpi_1)$ and $\Cbf_{n,\ell_2}(\varpi_2)$ are known then $\Cbf_{n,\ell_1\ell_2}(\varpi_1 + \varpi_2)$ holds. This is of course only interesting if $\varpi_1+ \varpi_2 < 1/2$. 

Finally,  for any prime $\ell | \, |\Cl_K|$, we have 
$
|\Cl_K[\ell^k]| \leq |\Cl_K[\ell]|^{k} 
$ for any integer $k \geq 1$.
(If $\ell \ndiv |\Cl_K|$ we can say this trivially holds, since both sides are 1.)
Indeed, more generally, for any $k' \geq k \geq 1$, 
 $|\Cl_K[\ell^{k'}]| \leq |\Cl_K[\ell^k]|^{k'/k}.
$
Thus if $\Cbf_{n,\ell^k}(\varpi)$ is known for some $k \geq 1$, then $\Cbf_{n,\ell^{k'}}(k'\varpi/k)$ holds for all $k' \geq k$. Again, this is only interesting if $k'\varpi/k <1/2$. As a consequence, if for a given degree $n$, Conjecture \ref{conj_class} holds for a prime $\ell$, then for all $k \geq1$, $|\Cl_K[\ell^k]| \ll_{n,\ell,k,\ep} D_K^\ep$ for all $\ep>0$.

\subsection{Pointwise results, prime and prime power $\ell$} 
We remarked that Gauss proved $\Cbf_{2,2}(0)$; from the discussion above, $\Cbf_{2,2^k}(0)$ follows for all $k \geq 1$. For  $n=2$ and $\ell=3$, initial progress occurred in \cite{HelVen06, Pie05, Pie06} and the current record of $\Cbf_{2,3}(1/3)$ is held by \cite{EllVen07}, which also proved 
 $\Cbf_{3,3}(1/3)$ and
 $ \Cbf_{4,3}(1/2-\del)$ for a small $\del>0$ (Ellenberg, personal communication, computed one can take $\del=1/168$ if $K$ is non-$D_4$). Bhargava et al.  \cite{BSTTTZ17} have established
$\Cbf_{n,2}(0.2784...)$ for $n=3,4$ and $\Cbf_{n,2}(1/2 - 1/2n)$ where $n \geq 5$, the first  results to hold for all fields of arbitrarily high degree.

\subsection{Pointwise results, composite $\ell$}  
By multiplicativity, combining Gauss's work with $\Cbf_{2,3}(1/3)$ thus implies $\Cbf_{2,2^k 3}(1/3)$ for all $k \geq 1$.

\subsection{Results conditional on GRH}\label{sec_GRH}
Much recent progress has relied on finding sufficiently many ``small'' rational primes that split completely within $K$, based on an insight credited to K. Soundararajan and P. Michel, and codified in work of Helfgott and Venkatesh \cite{HelVen06} and Ellenberg and Venkatesh \cite{EllVen07}, from which we quote the following:
\begin{letterthm}[{\cite[Lemma 2.3]{EllVen07}}]\label{thm_EV}
If a field $K$ of degree $n$ simultaneously  has the properties that 
\begin{enumerate}
\item the minimum multiplicative Weil height of a generating element of $K$ is large, say 
\beq\label{Weil}
 \inf \{ H_K(\al) : K = \Q(\al) \} \gg D_K^\eta, 
 \eeq
\item and $\gg M$ rational primes $p \leq D_K^{\eta/\ell}$ split completely in $K$, 
\end{enumerate}
then $|\Cl_K[\ell]| \ll_{n,\ell,\eta,\ep} D_K^{1/2 + \ep}M^{-1}$, for every $\ep>0$. 
\end{letterthm}

The idea underlying this theorem can be most easily seen in the case of $K$ an imaginary quadratic field of discriminant $-D$ and $\ell$ an odd prime. Letting $H = \Cl_K[\ell]$, we of course have 
\beq\label{cosets}
|H| = |\Cl_K| / [ \Cl_K:H],
\eeq
 so that to show that $|H|= |\Cl_K[\ell]|$ is small it suffices to show that $H$ has many cosets in $\Cl_K$. Suppose that $p_1,p_2$ are distinct rational primes with $p \ndiv 2  D$  that split in $K$ as $p_1 = \pfrak_1 \pfrak_1^\sig$, $p_2 = \pfrak_2 \pfrak_2^\sig$. If $\pfrak_1 H = \pfrak_2H$ then it follows that $\pfrak_1 \pfrak_2^\sig \in H$, so that $(\pfrak_1\pfrak_2^\sig)^\ell$ is a principal ideal, say $((u + v \sqrt{-D})/2)$ for some rational integers $u,v$. Upon taking norms, this shows that 
 \beq\label{key_eqn}
  4(p_1p_2)^\ell = u^2 + Dv^2
  \eeq
  for some rational integers $u,v$. Now if we had specified to begin with that $p_1,p_2 < (1/4)D^{1/(2\ell)}$, such a relation could not hold (since it would impose $v=0$ and $4(p_1p_2)^\ell$ cannot be a square of an integer), and we must have that $\pfrak_1$ and $\pfrak_2$ represent different cosets of $H$.
  
Now suppose we have $M$ distinct rational primes $p_1,\ldots, p_M < (1/4) D^{1/(2\ell)}$ with $p \ndiv 2D$, each of which splits in $K$. By the above argument, the cosets $\pfrak_jH$ must all be distinct for $j=1,\ldots, M$, so that by (\ref{cosets}), for every $\ep>0$,
\[ |\Cl_K[\ell]| \leq |\Cl_K| M^{-1} \ll_\ep D^{1/2+\ep} M^{-1}.\]
The fully general case of considering all degree $n$ fields requires much more sophisticated machinery; see \cite{EllVen07}.

How can we apply Theorem \ref{thm_EV} in practice? Regarding the parameter $\eta$ in (\ref{Weil}), the lower bound $\eta \geq 1/(2(n-1))$ is known for  all degree $n$ fields (see \cite[Lemma 2.2]{EllVen07} or \cite[Thm. 1]{Sil84}). Although the Chebotarev density theorem shows that a positive proportion of primes split completely in $K$, it is notoriously difficult to  find \emph{small} split primes (i.e. to show $M \gg D_K^{\eta/\ell}/\log D_K^{\eta/\ell}$, which is expected to be true). In general this can only be done effectively within a certain field $K$ by assuming GRH for the Dedekind zeta function associated to the Galois closure $\tilde{K}$, which implies an effective version of the Chebotarev density theorem, with a good explicit error term. Taken together, these considerations lead to   Ellenberg and Venkatesh's result in \cite{EllVen07} that $\Cbf_{n,\ell}(1/2 - 1/(2\ell(n-1)))$ holds for all $n,\ell$ if GRH is assumed; we will call this GRH-quality savings. This opens several avenues for possible progress, such as: first, try to remove the dependence on GRH, and second, try to improve the benchmark of GRH-quality savings, possibly by showing that $\eta$ can be taken to be larger in (\ref{Weil}).

\subsection{Average and ``almost all'' results}\label{sec_review_avg}
Recent work has been able to obtain savings on $|\Cl_K[\ell]|$ as strong as the GRH-quality savings described above (or stronger), but without assuming GRH, at the cost of showing that it holds only for almost all fields within a certain family.
This includes remarks of Soundararajan \cite{Sou00} on imaginary quadratic fields, and also recent work of Heath-Brown and the first author \cite{HBP17}. The latter work even does better than GRH-quality savings, by proving for each prime $\ell \geq 5$ the unconditional bound 
$|\Cl_K[\ell]| \ll_{\ell,\ep}  D_K^{1/2 - 3/(2\ell+2)+\ep}
$ for all but a possible density zero family of imaginary quadratic fields. This is achieved by nontrivial use of averaging over the discriminant.

 For each degree $n \leq 5$, Ellenberg and the first and third authors \cite{EPW17} developed a new combination of field-counting and sieving to show that GRH-quality savings on $|\Cl_K[\ell]|$ holds unconditionally
 for all but a possible density zero exceptional family of degree $n$ extensions of $\Q$ (non-$D_4$ in the case $n=4$).  This work relied in particular on precise counts for fields with certain properties, such as a count for fields in which a fixed prime splits completely, where the count has a power-saving error term with explicit dependence on the prime. The strategy of \cite{EPW17} has also been adapted in work of Frei and Widmer \cite{FreWid17} to prove results on average over the special family of totally ramified cyclic extensions of any fixed number field. 
In general, for degrees $n \geq 6$, methods do not yet exist to count all number fields of degree $n$ precisely enough for the strategies of \cite{EPW17} to be carried out.
 
 Another approach is to tackle the effective Chebotarev density theorem directly, without assuming GRH. This has been carried out by the authors in \cite{PTBW20}, which proves a new effective Chebotarev density theorem valid for almost all members of certain families of fields, and consequently proves GRH-quality savings hold for $\ell$-torsion, for all $\ell \geq 1$, in almost all such fields. (The families of fields are too numerous to describe here, but include for example:  totally ramified cyclic extensions of $\Q$, in which case the result is unconditional;  degree $n$ $S_n$-fields with squarefree discriminant, conditional for $n \geq 5$ on the strong Artin conjecture and certain counts for number fields; degree $n$ $A_n$-fields, conditional on the strong Artin conjecture.) We note that this approach has recently been adapted by An \cite{An18} to certain families of $D_4$-quartic fields. We note that an interesting contrast of \cite{PTBW20} to \cite{EPW17} is that while it also depends heavily on counting number fields, it only requires rough counts. That is, given a family of fields, at a key step in \cite{PTBW20} it suffices that one can show that at most $\ll X^\be$ fields in the family have any fixed discriminant $X$, and at least $\gg X^\al$ fields in the family have bounded discriminant $\leq X$, where $\al>\be$. The fact that such rough counts suffice allows \cite{PTBW20} to prove results for many new families of fields. See also Remark \ref{remark_deg_n} on the surprising efficiency of this method.

As mentioned above, the bound  $\eta \geq 1/(2(n-1))$ for the exponent $\eta$ in the expression (\ref{Weil}) in property (1)  is uniformly true for all degree $n$ extensions. For fields in which $\eta$ can be taken to be larger, one could also increase the savings in the $\ell$-torsion count by increasing $M$ in property (2). This interesting strategy for improving the very notion of ``GRH-quality savings'' has been proposed by Ellenberg \cite{Ell08} and taken up in recent work of Widmer \cite{Wid18}. One underlying question is which families can admit an improved uniform lower bound for $\eta$, or an improved lower bound on $\eta$ for almost all fields in the family. Very recently Frei and Widmer \cite{FreWid18x} have strengthened this approach, and their ideas combined with the methods of \cite{EPW17, PTBW20} improve the average upper bounds in certain results of those papers. For example, they can improve the relevant exponent $1/2 -1/(2\ell(n-1))$ in \cite{EPW17} for $n \geq 5$ to $\approx 
1/2 - 1/(\ell n)$; see \cite[Thm. 1.2]{FreWid18x}.

The above results apply to average or ``almost all'' results. But if we turn to pointwise results, is it even true that all fields of degree $n$ should have $\eta$ strictly larger than $1/(2(n-1))$? In some restricted families of degree $n$ fields, no such improvement is possible; see \cite{FreWid17}. Such considerations of enlarging $\eta$ relate also to conjectured uniform upper bounds such as $\eta \leq 1/2$, suggested by  Ruppert \cite{Rup98} (and proved to hold for ``almost all'' fields in many families of fields, by work of the authors \cite{PTBW20}). If true, this restriction on $\eta$ would imply a limit for how far the   ``small split prime approach'' can go in bounding $\ell$-torsion (e.g. the best possible would be $\Cbf_{n,\ell}(1/2 - 1/2\ell)$), without additional information or new ideas.

\subsection{Higher moments}
Theorem \ref{thm_CLM_torsion} makes a clear case for studying higher moments of $\ell$-torsion over families of fields. The first paper to do so was by Heath-Brown and the first author \cite{HBP17}, which proved nontrivial upper bounds for $k$-th moments of $\ell$-torsion in imaginary quadratic fields for odd primes $\ell$. Very recent work of Frei and Widmer \cite{FreWid18x} now makes progress on upper bounds for $k$-th moments in certain other settings, e.g. for $\ell$-torsion in (imaginary and real) quadratic fields, or cyclic cubic fields, by simultaneously incorporating ideas for improving (on average) the lower bound in (\ref{Weil})  into the settings of certain methods in \cite{EPW17, PTBW20}.

\subsection{Asymptotics} For recent work on asymptotics, see \S \ref{sec_CLM_predictions}.

\section{Appendices}\label{sec_appendices}
\subsection{Further remarks on moment bounds in a general setting}\label{sec_meas_space}

In the context of Theorem \ref{thm_CLM_torsion} we argued explicitly using the function $V: K \mapsto |\Cl_K[\ell]|$ mapping fields (ordered by discriminant) to positive real numbers. This was to make the argument clear in a setting where there could be more than one field with a fixed discriminant, and where the underlying finite set $\Fscr(X)$ changes as $X$ grows. We include further general remarks here, which clarify whether one needs uniformity in the constants appearing in the moments, as we anticipate there will be further applications for this perspective.

Let $(\Mcal,\mu)$ be a ($\sigma$-finite) measure space as before and consider the spaces $L^p(E)$ for $\mu(E)<\infty$ and $1 \leq p \leq \infty$. If $f \in L^q(E)$ then $f \in L^p(E)$ for all $1 \leq p \leq q$; indeed H\"older's inequality shows that 
\[ \|f \|_{L^p(E)} \leq \mu(E)^{1/p - 1/q} \|f\|_{L^q(E)}.\]
In particular, if $f \in L^\infty(E)$ then $f \in L^p(E)$ for all $1 \leq p \leq \infty$ with 
\beq\label{gen_infty}
\|f\|_{L^p(E)} \leq \mu(E)^{1/p} \|f\|_{L^\infty(E)}.
\eeq

We are interested in the converse direction. Let us make the strong assumption that there exist  constants $B,M$ such that for an infinite sequence of arbitrarily large $p \geq 1$ there exists a constant $c_p \leq M$ such that for all $f \in L^p(E)$,
\beq\label{gen_assp} \|f\|_{L^p(E)} \leq c_p \mu(E)^{B/p}.
\eeq
Then we claim $f \in L^\infty(E)$ and moreover $\|f\|_{L^\infty(E)} \leq M$.

Indeed, to show $f \in L^\infty(E)$ we must show there exists $C$ such that $\mu(\{ x : |f(x)| > C\})=0$. By Chebyshev's inequality (\ref{int_id}) and our assumption (\ref{gen_assp}), for every $\al>0$,
\[ \mu(\{ x \in E : |f(x)| > \al\}) \leq \frac{c_p^p}{\al^p} \mu(E)^B \]
for an infinite sequence of arbitrarily large $p \geq 1$. Let us choose $\al_0$ sufficiently large such that $\al_0 >M$. Then given any $\ep'>0$, we may take $p$ sufficiently large (depending on $M, \al_0, \mu(E), B$) that 
 \[\mu(\{ x \in E: |f(x)| > \al_0\}) \leq (M/\al_0)^{p} \mu(E)^B < \ep'. \]
Thus we may conclude that $f \in L^\infty(E)$, with $\|f\|_{L^\infty(E)} \leq M$. 
Consequently, (\ref{gen_infty}) holds for all $1 \leq p \leq \infty$. In particular, this shows that if the apparently weaker estimate (\ref{gen_assp}) holds for arbitrarily large $p$, this implies the apparently stronger estimate (\ref{gen_infty}); that is to say, assuming that (\ref{gen_assp}) holds for some $B$ and an infinite sequence of arbitrarily large $p$ is equivalent to assuming it holds for $B=1$ and all $1 \leq p \leq \infty$. 

These ideas apply to the case where $E = \{1,2, 3, \ldots, X\}$ and $\mu$ is counting measure, so that $\mu(E)=X$. 
They can also be adapted to the setting in which we are studying moments of a function $V$ mapping fields (ordered by discriminant, for example) to positive real numbers. Let $X$ be fixed. If it is known that there is some $\al$ and some $d_k \leq M^k$ for a uniform constant $M$ such that
\[ 
\| V\|_{\ell^k(\Fscr(X))}^k =  \sum_{K \in \Fscr(X)} V(K)^k \leq d_k |\Fscr(X)|^{\al}
\]
holds for a sequence of arbitrarily large $k$,
then this holds  with $\al=1$ for all $1 \leq k \leq \infty$, and indeed these two statements are equivalent to each other, and equivalent to the statement that $\|V\|_{\ell^\infty(\Fscr(X))} \leq M$. 

Above, we made the assumption in (\ref{gen_assp}) that there exists a uniform upper bound $M$ such that $c_p \leq M$ for all $p$ in  the sequence of arbitrarily large $p$. (Without this assumption, the $L^p$ norms can blow up and  the claim $f \in L^\infty(E)$ need not hold.)
However, the reader may note that in our use of (\ref{CL_k_upper}) in Theorem \ref{thm_CLM_torsion} we made no such assumption of an upper bound on the constant $c_{n,\ell,k,\al}$ being uniform in $k$. Indeed, we do not require such an assumption because in Theorem \ref{thm_CLM_torsion} we   prove not a uniform upper bound but merely show the growth with respect to $X$ is slow as $X \maps \infty$. 

In the language of the model setting above,  let $\Mcal=\Z$, $E=(X/2,  X] \intersect \Z$ for  $X \geq 1$  and let $\mu$ be counting measure. Let $f \in \ell^k((X/2,X]\intersect \Z)$ be given. Assume that there is some non-decreasing  function $F \geq 1$ such that $|f(x)| \leq F(x)$ for all $x$. Also assume 
 that there exists a  constant $B$ and an infinite sequence of arbitrarily large $k \geq 1$ with corresponding constants $c_k$  such that for all $X \geq 1$,
\[ \|f\|_{\ell^k((X/2,X]\intersect \Z)} \leq c_k X^{B/k}.
\]
Then for any fixed $\del>0$, for every $X \geq 1$,
\[ |\{ x \in (X/2,X] \intersect \Z : |f(x)| > x^\del \} |\leq \frac{c_k^k }{(X/2)^{\del k}} X^{B/k}.\]
Thus for any fixed $\del$, regardless of how big each constant $c_k$ is, we may take a fixed $k_0$ sufficiently large (determined only by $\del, B$) that $B/k_0 - \del k_0 < 0$, and consequently $c_{k_0}^{k_0} 2^{\del k_0} X^{B/k_0 - \del k_0} <1$ for $X $ sufficiently large with respect to $\del, B, c_{k_0}, k_0$. The set of constants $c_k$ was determined by the function $f$, so we will denote this dependence as a dependence on $f$.
This means that there exists some constant $X_0(f,\del,B)$ such that for $X \geq X_0(f,\del, B)$, the set $\{ x \in (X/2,X]\intersect \Z: |f(x)| > x^\del \} $ is empty. In other words, for any $\del>0$, the only $x$ such that  $|f(x)|> x^\del$ must have $x \leq X_0(f,\del,B)$ and hence $|f(x)| \leq F(x) \leq F(X_0(f,\del,B))$, which is a constant $\geq 1$ depending only on $f, \del, B$. In total, we have proved that given any $\del>0$ there exists a constant $C_\del = C_\del(f,B)$ such that for all $x \in \Z_{\geq 1}$, $|f(x)| \leq C_\del x^\del$ (it suffices to take $C_\del=F(X_0(f,\del,B))$). This argument, adapted to families of fields, underlies Theorem \ref{thm_CLM_torsion}.

\begin{remark}\label{remark_1}
The above discussion also verifies our earlier remark that the following statements are equivalent: (i)  (\ref{CL_k_upper}) holds (for all $X \geq 1$) for a particular fixed $\al$ and an infinite sequence of arbitrarily large $k$;  (ii)  (\ref{CL_k_upper}) holds (for all $X \geq 1$) for all $\al = 1 + \ep_0$ with $\ep_0>0$  and all integers $k \geq 1$. Certainly the   statement (ii) implies statement (i). In the other direction, the proof we gave for Theorem \ref{thm_CLM_torsion} (with $k$ ranging over all integers $k \geq 1$) adapts easily to the case where $k$ ranges over an infinite sequence of arbitrarily large $k$, and by this method we can prove that statement (i) implies that for every $\ep_0> 0$, $|\Cl_K[\ell]| \leq C_{\ep_0} D_K^{\ep_0}$ holds uniformly for $K \in \Fscr$. Averaging this pointwise upper bound then implies statement (ii), completing the equivalence.
\end{remark}

\subsection{Counting quartic $D_4$-fields with fixed discriminant}\label{sec_app_D4}
We briefly prove that property $\Dbf_4(D_4;0)$ holds; that is, for every $D \geq 1$ and for every $\ep>0$ there are at most $\ll_\ep D^\ep$ quartic $D_4$-fields of discriminant $D$. We follow \cite[\S 4]{Bai80}, except we fix a discriminant $D$ instead of considering all discriminants up to $X$.
Any quartic $D_4$-field $K_4$ is a quadratic extension of a quadratic field $K_2$.  We have that $K_4$ corresponds to a quadratic ray class character for $K_2$ of conductor $\mathfrak{d}$ with finite part $\mathfrak{d}^*=\Disc(K_4/K_2)$.  The form of such characters has been determined in \cite[Lemma 17]{Bai80}, and each character is a product of a character on $(\O_{K_2}/\mathfrak{d^*})^\times$, a character on the class group of $K_2$, and a character on signature (see \cite{Bai80}[Eqn. (4)]).  There are at most 4 characters of signature, and $h_2(K_2)$ characters of the class group, which is $\ll_\ep (\Disc K_2)^{\ep}$ for every $\ep>0$ by genus theory (\ref{Gauss_2}).  It follows from the proof of \cite[Lemma 18]{Bai80}
that given any $C$, there are at most $O(2^{\omega(C)})$ ideals $\mathfrak{d^*}$ and possible characters on $(\O_{K_2}/\mathfrak{d^*})^\times$ with  $N_{K_2/\Q} (\mathfrak{d}^*)=C$.
Note $\Disc K_4=(\Disc K_2)^2 \cdot N_{K_2/\Q} (\Disc(K_4/K_2))$.
Thus, given $D$, in order to count the quartic $D_4$-fields of discriminant $D$, we sum for each $d|D$, the number of quartic $D_4$-fields $K_4$ that contain the quadratic field of discriminant $d$, and we find there are at most
$$
O_{\ep}\left(\sum_{d|D} d^\ep 2^{\omega(D/d^2)} \right)=O_{\ep}(D^{\ep})
$$
quartic $D_4$-fields of discriminant $D$.

\section*{Acknowledgements}
The authors thank J. Ellenberg,  J. Wang,  B. Alberts, and the referee for  insightful comments that improved and clarified the content and exposition.
Pierce has been partially supported by NSF CAREER grant DMS-1652173, a Sloan Research Fellowship, a Birman Fellowship, and as a von Neumann Fellow at the Institute for Advanced Study, by the Charles Simonyi Endowment and NSF Grant No. 1128155. Pierce thanks the Hausdorff Center for Mathematics for hospitality during visits as a Bonn Research Fellow.
Turnage-Butterbaugh is partially supported by NSF DMS-1902193 and NSF DMS-1854398 FRG, and thanks the Mathematical Sciences Research Institute (NSF DMS-1440140) and the Max Planck Institute for Mathematics for support and hospitality during portions of this work.
Wood was partially supported by an American Institute of Mathematics Five-Year Fellowship, a Packard Fellowship for Science and Engineering, a Sloan Research Fellowship, National Science Foundation grant DMS-1301690 and CAREER grant DMS-1652116, and a Vilas Early Career Investigator Award.  Wood thanks Princeton
University for its hospitality during Fall 2018.

\bibliographystyle{alpha}
\bibliography{NoThBibliography}
\label{endofproposal}

\end{document}